\documentclass[a4paper,10pt,leqno]{amsart}
        \title{Obstructions to fibering a manifold}

       \author{F.~T.~Farrell}
       \address{Department of Mathematics\\Suny, Binghamton\\ Ny, 13902, \\U.S.A.}
       \email{farrell@math.binghamton.edu}
       \author{Wolfgang L\"uck}
       \address{Westf\"alische Wilhelms-Universit\"at M\"unster\\
               Mathematisches Institut\\
               Einsteinstr.~62,
               D-48149 M\"unster, Germany}
       \email{lueck@math.uni-muenster.de}
       \urladdr{http://www.math.uni-muenster.de/u/lueck}
       \author{Wolfgang Steimle}
       \address{Westf\"alische Wilhelms-Universit\"at M\"unster\\
               Mathematisches Institut\\
               Einsteinstr.~62,
               D-48149 M\"unster, Germany}
       \email{steimle@math.uni-muenster.de}
       \date{August 2009}
       \dedicatory{Dedicated to Bruce Williams  on the occasion of his 
                     60th birthday}

\usepackage{hyperref}
\usepackage{calc}
\usepackage{enumerate,amssymb}
\usepackage[arrow,curve,matrix,tips,2cell]{xy}
  \SelectTips{eu}{10} \UseTips
  \UseAllTwocells

\DeclareMathAlphabet{\matheurm}{U}{eur}{m}{n}

\DeclareMathOperator{\can}{can}

\DeclareMathOperator{\cok}{cok}

\DeclareMathOperator{\con}{con}
\DeclareMathOperator{\cone}{cone}
\DeclareMathOperator{\cyl}{cyl}
\DeclareMathOperator{\FIB}{FIB}
\DeclareMathOperator{\fib}{fib}

\DeclareMathOperator{\hofib}{hofib}

\DeclareMathOperator{\id}{id}

\DeclareMathOperator{\map}{map}

\DeclareMathOperator{\pr}{pr}

\DeclareMathOperator{\topo}{Top}

\DeclareMathOperator{\Wh}{Wh}
\DeclareMathOperator{\Wall}{Wall}

  \newcommand{\IC}{\mathbb{C}}

  \newcommand{\IP}{\mathbb{P}}
  
  \newcommand{\IR}{\mathbb{R}}

  \newcommand{\IZ}{\mathbb{Z}}

\newcounter{commentcounter}

\theoremstyle{plain}
\newtheorem{theorem}{Theorem}[section]

\newtheorem{lemma}[theorem]{Lemma}
\newtheorem{corollary}[theorem]{Corollary}
\newtheorem{proposition}[theorem]{Proposition}

\newtheorem{definition}[theorem]{Definition}
\theoremstyle{definition}

\newtheorem{example}[theorem]{Example}
\newtheorem{question}[theorem]{Question}

\newtheorem{remark}[theorem]{Remark}
\newtheorem{notation}[theorem]{Notation}

\theoremstyle{remark}

\makeatletter\let\c@equation=\c@theorem\makeatother

\newcommand{\version}[1]                       
{\begin{center} Last edited on #1\\
Last compiled on \today\\
name of tex-file: \jobname
\end{center}
}

\newcommand{\xycomsquare}[8]                
{$$\xymatrix{#1 \ar[r]^-{#2} \ar[d]^{#4} &
    #3 \ar[d]^{#5}  \\
    #6\ar[r]^-{#7} & #8 }$$}

\newcommand{\sections}[2]{\Gamma\biggl(\begin{array}{c} {#1}\\ \downarrow\\ {#2}\end{array} \biggr)}

\begin{document}

\maketitle

\begin{abstract}
Given a map $f \colon M \to N$ of closed topological manifolds we define
torsion obstructions whose vanishing is a necessary condition for 
$f$ being homotopy equivalent to a projection of a locally trivial fiber bundle.
If $N = S^1$, these torsion obstructions are identified with the ones due to 
Farrell~\cite{Farrell(1971)}.
\end{abstract}


\typeout{---------------------------- Introduction
  ------------------------------------}
\section*{Introduction}
\label{sec:introduction}

Given a map $f \colon M \to N$ of closed topological manifolds, we define
torsion obstructions whose vanishing is a necessary condition for 
$f$ being homotopy equivalent to a projection of a locally trivial fiber bundle.
If $N = S^1$, these torsion obstructions are identified with the ones due to 
Farrell~\cite{Farrell(1971)}.

The basic idea of the construction is as follows. A simple structure
$\xi$ on a space $Y$ with the homotopy type of a finite $CW$-complex
is the choice of an equivalence class of  homotopy equivalences
$X \to Y$ with a finite $CW$-complex as domain, where we call two such maps
$f_i \colon X_i \to Y$ for $i = 1,2$ equivalent, if 
$f_2^{-1} \circ f_1 \colon X_1 \to X_2$
is a simple homotopy equivalence. The classical theory of Whitehead torsion
for homotopy equivalences between finite $CW$-complexes extends to homotopy
equivalences of space with simple structures.

Consider a map $f \colon M \to N$ of closed topological manifolds.
Suppose that $N$ is connected. 
Fix a base point $y$ for $N$. Assume that the
homotopy fiber $\hofib(f)$ of $f$ at $y$ has the homotopy type of a finite
$CW$-complex.   By inspecting the fiber transport of the fibration $\mu_f \colon
\FIB(f) \to N$ associated to $f$, one obtains a homomorphism from
$\pi_1(N,y)$ to the group of homotopy classes of self-homotopy
equivalences of the homotopy fiber.  
In the sequel $\pi(X)$ denotes the fundamental groupoid
of a space $X$ and $\Wh(\pi(X))$ the associated Whitehead group.
Pick a simple structure on the homotopy fiber  $\hofib(f)$. 
If we consider the image of the Whitehead torsion of
self-homotopy equivalences of the homotopy fiber under 
$\Wh\bigl(\pi(\hofib(f))\bigr) \to \Wh(\pi(M))$, we obtain a homomorphism 
$\pi_1(N,y) \to \Wh(\pi(M))$.  It defines an element
$$\Theta(f) \in H^1\bigl(N;\Wh(\pi(M))\bigr).$$
Its definition is independent of the choice of base point $y \in N$.

The element $\Theta(f)$ depends only on the homotopy class of $f$.  If
$f$ is the projection of a locally trivial fiber bundle, then the
fiber transport is given by homeomorphisms and by the topological
invariance of Whitehead torsion this implies $\Theta(f) = 0$.

From now on suppose $\Theta(f) = 0$. Assume for simplicity for the remainder of the
introduction that the Euler characteristic $\chi(N)$ of $N$ is zero.  Then
one can construct a preferred simple structure $\xi(\FIB(f))$ on $\FIB(f)$.
This is obvious if the fibration $\FIB(f)$  is trivial
since the cartesian product of a homotopy equivalence of finite
$CW$-complexes with a finite $CW$-complex
of Euler characteristic zero is simple.
The general case is done by induction over the cells of $Y$
using a construction of a 
pushout simple structure and the fact that a fibration over
$D^n$ is homotopically trivial, where $Y \to N$ 
is some representative of the preferred
simple structure on the closed topological manifold $N$.
Let $\mu_f \colon \FIB(f) \to M$ be the canonical homotopy
equivalence.  Since $M$ is a topological manifold, it carries a
preferred simple structure. Hence the Whitehead torsion $\tau(\mu_f)$
of $\mu_f$ makes sense, and we define a second invariant
$$\tau_{\fib}(f) := \tau(\mu_f) \in \Wh(\pi(M)).$$
The element $\tau_{\fib}(f)$ depends only on the homotopy class of
$f$.  If $f$ is the projection of a locally trivial fiber bundle, 
then $\tau_{\fib}(f) = 0$.

As an illustration we study the case, where the base space is $S^1$
and identify our invariants with the obstruction of fibering $M$ over
$S^1$ due to Farrell~\cite{Farrell(1971)}.  Notice that in this case
Farrell~\cite{Farrell(1971)} shows that the vanishing of the
obstructions imply that $f$ homotopic to the projection of a locally trivial
fiber bundle provided $\dim(M) \ge 5$.  For an arbitrary closed
manifold $N$ as target of $f$ the vanishing of these obstructions will
be necessary but not sufficient for $f$ being homotopic to the
projection of a locally trivial fiber.

We give a composition formula for $\tau_{\fib}$.

We introduce Poincar\'e torsion which is the obstruction for a finite
Poincar\'e complex to be homotopy equivalent to a simple
Poincar\'e complex. 

Finally we briefly give a connection to the parametrized $A$-theoretic
characteristic due to Dwyer-Weiss-Williams
\cite{Dwyer-Weiss-Williams(2003)}  and discuss some open questions.

The paper was supported by the Sonder\-forschungs\-be\-reich
478 -- Geometrische Strukturen in der Mathematik --, the
Graduiertenkolleg -- Analytische Topologie und Metageometrie -- and the
Max-Planck-Forschungspreis and the Leibniz-Preis of the second  author. It was also partially supported by an NSF grant
of the first author.

The third author would like to thank Bruce Williams for helpful explanations concerning the $A$-theory characteristic.

The paper is organized as follows:

\begin{tabular}{ll}%
\ref{sec:Simple_structures_and_Whitehead_torsion}.
&
Simple structures and Whitehead torsion
\\%
\ref{sec:Fibrations}.
&
Fibrations
\\%
\ref{sec:The_simple_structure_on_a_total_space_of_a_fibration}.
&
The simple structure on a total space of a fibration
\\%
\ref{sec:Turning_a_map_into_a_fibration}.
&
Turning a map into a fibration
\\%
\ref{sec:Fiber_torsion_obstructions}.
&
Fiber torsion obstructions
\\%
\ref{sec:base_space_S1}.
&
Base space $S^1$
\\%
\ref{sec:Gluing_h-cobordisms}.
&
Gluing $h$-cobordisms
\\%
\ref{sec:comparison_with_Farrells_obstruction_over_S1}.
&
Comparison with Farrell's obstruction over $S^1$
\\%
\ref{sec:a_composition_formula}.
&
A composition formula
\\%
\ref{sec:Poincare_torsion}.
&
Poincar\'e torsion
\\%
\ref{sec:Connection_to_the_parametrized_A-theory_characteristic}.
&
Connection to the parametrized $A$-theory characteristic
\\%
\ref{sec:Some_questions}.
&
Some questions
\\%
& References
\end{tabular}

\typeout{----------------------- Section 1 ------------------------}

\section{Simple structures and Whitehead torsion}
\label{sec:Simple_structures_and_Whitehead_torsion}

In this section we extend the definition of the  Whitehead torsion
of homotopy equivalences between finite $CW$-complexes to
homotopy equivalences between more general spaces, namely, 
spaces with simple structures.

Let $Y$ be a space of the homotopy type of a finite $CW$-complex.  We
call two maps $f_1\colon X_1 \to Y$ and $f_2\colon X_2 \to Y$ with
finite $CW$-complexes as source and $Y$ as target \emph{simply
  equivalent} if the Whitehead torsion $\tau(f_2^{-1} \circ f_1 \colon
X_1 \to X_2) \in \Wh(\pi(X_2))$ vanishes.  (For the notion of
Whitehead torsion and Whitehead group we refer to \cite{Cohen(1973)}.)

\begin{definition} \label{def:simple_structure} A \emph{simple
    structure} $\xi$ on a space $Y$ with the homotopy type of a finite
  $CW$-complex is a choice of a simple equivalence class of homotopy
  equivalences $u\colon X \to Y$ with a finite $CW$-complex as source
  and $Y$ as target. If $Y$ is a finite $CW$-complex, we refer to the
  simple structure represented by $\id_Y$ as 
  \emph{canonical simple structure} $\xi_{\can}(Y)$ on $Y$.
\end{definition}

Let \xycomsquare{Y_0}{i_1}{Y_1}{i_2}{j_1}{Y_2}{j_2}{Y} be a pushout of
spaces with $i_1\colon Y_0 \to Y_1$ a cofibration.  Suppose that $Y_i$
has the homotopy type of a finite $CW$-complex and comes with a simple
structure $\xi_i$ for $i=0,1,2$.  Then $Y$ has the homotopy type of a
finite $CW$-complex and there is a preferred simple structure $\xi$ on
$Y$ which we will call \emph{the pushout simple structure} and which
is constructed as follows.  Choose a pushout of finite $CW$-complexes
\begin{equation}\label{cellular_pushout}
  \xymatrix{X_0 \ar[r]^{a_1} \ar[d]^{a_2} & X_1 \ar[d]^{b_1}\\
    X_2 \ar[r]^{b_2} & X}
\end{equation}
together with homotopy equivalences $u_i\colon X_i \to Y_i$
representing $\xi_i$ for $i=0,1,2$ such that the maps $a_1$ and $b_2$
are inclusions of $CW$-subcomplexes, the maps $a_2$ and $b_1$ are
cellular and the $n$-skeleton $X_n$ of $X$ is the subspace
$b_1\bigl((X_1)_n\bigr) \cup b_2\bigl((X_2)_n\bigr)$ for every $n \ge
-1$.  The pushout property yields a map $u\colon X \to Y$ which is a
homotopy equivalence. Let the pushout simple structure $\xi$ be the
one represented by $u$.

The proof that such a diagram \eqref{cellular_pushout} together with
maps $u_i$ exists and that $\xi$ only depends on $\xi_i$ and not on
the choice of $(X_i,u_i)$ can be found in \cite[page~74~ff.]{Lueck(1989)}.

Given two spaces $(X,\xi)$ and $(Y,\eta)$ with simple structures, the
\emph{product simple structure} $\xi\times \eta$ on $X \times Y$ is
represented by crossing some representative for $\xi$ with some
representative for $\eta$. This is well-defined since the product of
two simple homotopy equivalences between finite $CW$-complexes is again
a simple homotopy equivalence.

Given a homotopy equivalence $f\colon (X,\xi) \to (Y,\eta)$ of spaces
with simple structures, we define its Whitehead torsion
\begin{eqnarray}
  \tau(f) & \in & \Wh(\pi(Y))
  \label{definition_of_Whitehead_torsion_for_spaces_with_simple_structure}
\end{eqnarray}
by $v_*(\tau(v^{-1}\circ f \circ u))$, where $u\colon X' \to X$ and
$v\colon Y' \to Y$ are representatives of the simple structures,
$\tau(v^{-1}\circ f \circ u) \in \Wh(\pi(Y'))$ is the classical Whitehead
torsion of a homotopy equivalence of finite $CW$-complexes and
$v_*\colon \Wh(\pi(Y')) \to \Wh(\pi(Y))$ is the isomorphism induced by
$v$.  The standard properties of the Whitehead torsion of a homotopy
equivalence of finite $CW$-complexes carry over to homotopy
equivalences of spaces with simple structure. Namely, we get (see
\cite[(22.1), (23.1), (23.2)]{Cohen(1973)}, \cite[Theorem
4.33]{Lueck(1989)})

\begin{lemma}
  \label{lem:properties_of_Whitehead_torsion}

  \begin{enumerate}

  \item \label{lem:properties_of_Whitehead_torsion:homotopy_invariance}
    Homotopy invariance\\[1mm]
    Let $f,g\colon X \to Y$ be maps of spaces with simple structures.
    If $f$ and $g$ are homotopic, then 
    $$\tau(f) =\tau(g);$$

  \item \label{lem:properties_of_Whitehead_torsion:composition_formula}
    Composition formula\\[1mm]
    Let $f:X \to Y$ and $g\colon Y \to Z$ be maps of spaces with
    simple structures. Then 
    $$\tau(g \circ f) = \tau(g) + g_*(\tau(f));$$

  \item \label{lem:properties_of_Whitehead_torsion:sum_formula}
    Sum formula\\[1mm]
    Let\\[2mm]
    $\xymatrix{X_0 \ar[r]^-{a_1} \ar[d]^{a_2} &
      X_1 \ar[d]^{b_1}  \\
      X_2\ar[r]^-{b_2} & X }$ \quad \raisebox{-7mm}{and } \quad
    $\xymatrix{Y_0 \ar[r]^-{i_1} \ar[d]^{i_2} &
      Y_1 \ar[d]^{j_1}  \\
      Y_2\ar[r]^-{j_2} & Y }$
    \\[2mm]
    be pushouts of spaces with $a_1$ and $i_1$ cofibrations.  Let
    $j_0\colon Y_0 \to Y$ be $j_1 \circ i_1 = j_2 \circ i_2$.  Suppose
    that all spaces come with simple structures such that $X$ and $Y$
    carry the pushout simple structure.  Let $f_i\colon X_i \to Y_i$
    be homotopy equivalences for $i=0,1,2$ such that $f_1 \circ a_1 =
    i_1 \circ f_0$ and $f_2 \circ a_2 = i_2 \circ f_0$. Let $f\colon X
    \to Y$ be the map induced by the pushout property.

    Then $f$ is a homotopy equivalence and
$$\tau(f) = (j_1)_*(\tau(f_1)) + (j_2)_*(\tau(f_2)) - (j_0)_*(\tau(f_0));$$

\item \label{lem:properties_of_Whitehead_torsion:product_formula}
  Product formula\\[1mm]
  Let $f_1\colon X_1 \to Y_1$ and $f_2\colon X_2 \to Y_2$ be homotopy
  equivalences of path-connected spaces with simple structures.  Equip
  $X_1 \times X_2$ and $Y_1\times Y_2$ with the product simple structures.
  Define $i_1\colon Y_1 \to Y_1 \times Y_2$ to be the inclusion
  of $Y_1$ into $Y_1 \times Y_2$ with respect to some base point $y_2
  \in Y_2$ and analogously define $i_2\colon Y_2 \to Y_1 \times Y_2$.

  Then
$$\tau(f_1 \times f_2) = \chi(Y_1) \cdot (i_2)_*(\tau(f_2))
+ \chi(Y_2) \cdot (i_1)_*(\tau(f_1)).$$
\end{enumerate}
\end{lemma}

\begin{remark} \label{rem:not-necessarily_cellular}
Let $X$ be a finite $CW$-complex. Consider any pushout describing
how $X_n$ is obtained from $X_{n-1}$ by attaching cells
$$\xymatrix@!C=6em{\coprod_{I_n} S^{n-1} \ar[r]^-{\coprod_{i \in I_n} q_i} \ar[d]^{} &
  X_{n-1} \ar[d]^{}  \\
  \coprod_{I_n} D^n \ar[r]^-{\coprod_{i \in I_n} Q_i} & X_{n} }$$
Equip the two upper corners and the left lower corner with the canonical simple structure
with respect to any $CW$-structure. Then the pushout simple structure on the right lower
corner agrees with the canonical simple structure with respect to any $CW$-structure.

This is obvious if we equip each $S^{n-1}$ with some finite $CW$-structure,
$D^n$ with the $CW$-structure which is obtained from the one on $S^{n-1}$  by attaching
one $n$-cell with the identity $S^{n-1} \to S^{n-1}$, we equip $X_{n-1}$ and $X_n$ with
given $CW$-structures and each map $q_i$ is cellular. The general case follows
using the cellular approximation theorem, the fact that changing the 
attaching maps by a homotopy does not change the simple homotopy type
(see~\cite[see~(7.1) on page~23]{Cohen(1973)})
and the topological invariance of Whitehead torsion
(see~\cite{Chapman(1974)}).
\end{remark}

\typeout{----------------------- Section 2 ------------------------}

\section{Fibrations}
\label{sec:Fibrations}

In this section we record some basis facts about fibrations.

Recall that a fibration $p\colon E \to B$ is a map which has the
\emph{homotopy lifting property}, i.e., for any homotopy $h\colon X
\times [0,1] \to B$ and map $f\colon X \to E$ with $p \circ f = h_0$
there is a homotopy $H\colon X \times [0,1] \to E$ satisfying $p \circ
H = h$ and $H_0 = f$, where here and in the sequel $h_t(x) := h(x,t)$
and $H_t(x) := H(x,t)$.  For general information about fibrations we
refer for instance to \cite[page~342~ff.]{Switzer(1975)},
\cite[I.7]{Whitehead(1978)}.  We mention that we will work in the
category of compactly generated spaces \cite{Steenrod(1967)},
\cite[I.4]{Whitehead(1978)}.  A map $(\overline{f},f) \colon p \to p'$
of fibrations $p\colon E \to B$ to $p'\colon E' \to B'$ consists of a
commutative diagram \xycomsquare{E}{\overline{f}}{E'}{p}{p'}{B}{f}{B'}
A homotopy $h\colon X \times [0,1] \to E$ is called a \emph{fiber
  homotopy} if $p \circ h$ is stationary, i.e., $p \circ h(x,t) =
p\circ h(x,0)$ for all $(x,t) \in X \times [0,1]$. Two maps
$f_0,f_1\colon X \to E$ with $p \circ f_0 = p \circ f_1$ are called
\emph{fiber homotopic} $f_0 \simeq_p f_1$ if there is a fiber homotopy
$h\colon X \times[0,1] \to E$ with $h_0 = f_0$ and $h_1 = f_1$.  A
\emph{fiber homotopy equivalence} from the fibration $p\colon E \to B$
to the fibration $p'\colon E' \to B$ over the same base is a map of
fibrations of the shape $(\overline{f},\id)\colon p \to p'$ such that
there exists a map of fibrations $(\overline{g},\id)$ with
$\overline{g} \circ \overline{f} \simeq_p \id$ and $\overline{f} \circ
\overline{g} \simeq_p \id$. The pullback of a fibration $p\colon E \to
B$ with a map $f\colon B' \to B$
\xycomsquare{f^*E}{\overline{f}}{E}{p_f}{p}{B'}{f}{B} is again a
fibration $p_f\colon f^*E \to B$. The elementary proof of the next
lemma can be found for instance in \cite[page~342~ff.]{Switzer(1975)}.

\begin{lemma} \label{lem:fiber_homotopies} Let $p\colon E \to B$ be a
  fibration.

  \begin{enumerate}
  \item \label{lem:fiber_homotopies:homotopies} Let $H\colon X \times
    [0,1] \to B$ be a homotopy $f_0 \simeq f_1\colon X \to B$.  Let
    $H'\colon f^*_0E \times [0,1] \to E$ be a solution of the homotopy
    lifting problem for $H\circ (p_{f_0} \times \id_{[0,1]}): f^*_0E
    \times [0,1] \to B$ and $\overline{f_0}\colon f_0^*E \to
    E$. Define $g_H\colon f_0^*E \to f_1^*E$ by $H_1'$ and $p_{f_0}$
    using the pullback property of $f_1^*E$.

    Then $(g_H,\id)\colon f_0^*E \to f_1^*E$ is a fiber homotopy
    equivalence and $H'$ is a homotopy $\overline{f_1} \circ g_H
    \simeq \overline{f_0}$.

  \item \label{lem:fiber_homotopies:homotopies_of_homotopies} Let
    $K\colon X \times [0,1] \to B$ be a second homotopy $f_0 \simeq
    f_1\colon X \to B$ and $M\colon X \times [0,1] \times [0,1]\to B$
    be homotopy relative $X \times \{0,1\}$ between $H$ and $K$. Then
    $M$ induces a fiber homotopy from $g_H$ to $g_K$.
  \end{enumerate}
\end{lemma}

Let $p\colon E \to B$ be a fibration. Denote by $F_b$ the fiber
$p^{-1}(b)$ for $b \in B$.  For any homotopy class $[w]$ of paths
$w\colon [0,1] \to B$ we obtain by Lemma~\ref{lem:fiber_homotopies} a
homotopy class $t([w])$ of maps $F_{w(0)} \to F_{w(1)}$ called the
\emph{fiber transport along w}.  If $v$ and $w$ are paths with $v(1) =
w(0)$, then $t([w]) \circ t([v]) = t([v \ast w])$. The constant path
$c_b$ induces the identity on $F_b$.  We mention that in the situation
of
Lemma~\ref{lem:fiber_homotopies}~\ref{lem:fiber_homotopies:homotopies}
for each $x \in X$ the map $F_{f_0(x)} \to F_{f_1(x)}$ induced by
$g_H$ represents the fiber transport along the path $H(x,-)$.

\begin{definition} \label{def:fiber_trivialization} Let $p\colon E \to
  B$ be a fibration, $f\colon X \to B$ be a map from a
  space $X$ to $B$ and $x \in X$ and $b \in B$.  Let $w\colon [0,1]
  \to B$ be a path from $b$ to $f(x)$.  A \emph{fiber trivialization}
  of $f^*E$ with respect to $(b,x,w)$ is a fiber homotopy equivalence
  $T\colon F_b \times X \to f^*E$ over $X$ such that the map $F_b \to
  F_{f(x)}$ induced by $T$ represents the fiber transport $t([w])$ for
  $p$ along $w$.
\end{definition}

Lemma~\ref{lem:fiber_homotopies} implies
\begin{lemma}
  \label{lem:fiber_trivializations}
  Consider the situation of Definition~\ref{def:fiber_trivialization}.
  Suppose additionally that $X$ is contractible. Then

  \begin{enumerate}

  \item \label{lem:fiber_trivializations:existence} The exists a fiber
    trivialization with respect to $(b,x,w)$;

  \item \label{lem:fiber_trivializations:uniqueness} Two fiber
    trivializations with respect to $(b,x,w)$ are fiber homotopic;

  \item \label{lem:fiber_trivializations:composition} Let $T_i\colon
    F_{b_i} \times X \to f^*E$ be a fiber trivialization with respect
    to $(b_i,x_i,w_i)$ for $i= 0,1$.  Choose a path $v\colon [0,1] \to
    X$ from $x_0$ to $x_1$.  Let $t\colon F_{b_0} \to F_{b_1}$ be a
    representative of the fiber transport of $p$ along $w_0 \ast f(v)
    \ast w_1^-$. Then we get a fiber homotopy
$$T_1 \circ (t \times \id_X) \simeq_{p_f} T_0;$$

\item \label{lem:fiber_trivializations:homotopy_invariance} Let
  $H\colon X \times [0,1] \to B$ be a homotopy $f_0 \simeq f_1$.  Let
  $v$ be the path in $B$ from $f_0(x)$ to $f_1(x)$ given by $h(x,-)$
  and let $w_0$ be a path from $b$ to $f_0(x)$. Put $w_1 = w_0 \ast
  v$.  Let $T_i\colon F_b \times X \to f_i^*E$ be the fiber
  trivialization for $f_i$ with respect to $(b,x,w_i)$ for $i =0,1$ and
  $g_H\colon f_0^*E \to f_1^*E$ be the fiber homotopy equivalence of
  Lemma~\ref{lem:fiber_homotopies}~\ref{lem:fiber_homotopies:homotopies}.
  Then we get a fiber homotopy over $X$
$$g_H \circ T_0 \simeq_{p_{f_1}} T_1.$$

\end{enumerate}
\end{lemma}

\typeout{----------------------- Section 3 ------------------------}
\section{The simple structure on a total space of a fibration}
\label{sec:The_simple_structure_on_a_total_space_of_a_fibration}

In this section we explain how the total space of a fibration inherits
a simple structure from the base space and the fiber.

\begin{definition} \label{def:spider} Let $B$ be a connected $CW$-complex with
  base point $b \in B$.  Denote by $I(B)$ the set of open cells of $B$
  and by $\dim(c)$ the dimension of a cell $c \in I(B)$.  A
  \emph{spider} at $b$ for $B$ is a collection of paths $w_c$ indexed
  by $c \in I(B)$ such that $w_c(0) = b$ and $w_c(1)$ is a point in the
  open cell $c$.
\end{definition}

Let $p\colon E \to B$ be a fibration such that $B$ is a path-connected
finite $CW$-complex and the fiber has the homotopy type of a finite
$CW$-complex.  Given a base point $b \in B$, a spider $s$ at $b$ and a
simple structure $\zeta$ on $F_b$, we want to construct a preferred
simple structure
\begin{eqnarray}\label{eq:simple_structure_on_total_space}
  & \xi(b,s,\zeta) \mbox{ on } E&
  \label{preferred_simple_structure_on_total_space}
\end{eqnarray}
as follows. Let $B_n$ be the $n$-skeleton of $B$ and $E_n =
p^{-1}(B_n)$.  We construct the preferred simple structure on $E_n$
inductively for $n = -1,0,1, \ldots $. The case $n = -1$ is trivial;
the induction step from $(n-1)$ to $n$ is done as follows. Choose a
pushout
$$\xymatrix@!C=6em{\coprod_{I_n} S^{n-1} \ar[r]^-{\coprod_{i \in I_n} q_i} \ar[d]^{} &
  B_{n-1} \ar[d]^{}  \\
  \coprod_{I_n} D^n \ar[r]^-{\coprod_{i \in I_n} Q_i} & B_{n} }$$


Choose for $i \in I_n$ the point $x_i \in D^n - S^{n-1}$ such that
$Q_i(x_i) = w_i(1)$, where $w_i$ is the path with $w_i(0) = b$
associated by the spider $s$ to the cell indexed by $i \in I_n$.  We
get from
Lemma~\ref{lem:fiber_trivializations}~\ref{lem:fiber_trivializations:existence}
a fiber trivialization $T_i\colon F_b \times D^n \to Q_i^*E$.  It
yields a homotopy equivalence of pairs
$$T_i\colon F_b \times (D^n,S^{n-1}) \to (Q_i^*E,q_i^*E).$$
Equip $Q_i^*E$ and $q_i^*E$ with the simple structures induced by $T_i$
from the product simple structure $(\zeta\times \xi_{\can}(D^n))$ on
$F_b \times D^n$ and and $(\zeta\times \xi_{\can}(S^{n-1}))$ on 
$F_b\times S^{n-1}$.

By induction hypothesis we have already constructed a simple structure on
$E_{n-1}$.
Since there is a pushout with a cofibration as
left vertical map (see~\cite[Lemma 1.26]{Lueck(1989)})
$$\xymatrix@!C=6em{\coprod_{I_n} q_i^*E \ar[r]^-{\coprod_{i \in I_n} 
\overline{q_i}} \ar[d]^{} &
    E_{n-1} \ar[d]^{}  \\
    \coprod_{I_n} Q_i^*E\ar[r]^-{\coprod_{i \in I_n} \overline{Q_i}} & E_n }$$
we can equip $E_n$ with the pushout simple structure.
Lemma~\ref{lem:fiber_trivializations}~\ref{lem:fiber_trivializations:uniqueness}
implies that the choice of
$T_i$ does not matter. 

Notice that the choice of the characteristic
maps $(Q_i,q_i)$ does not belong to the structure of a $CW$-complex.
Only the skeletal filtration $(X_n)_{n \ge -1}$
is part of the  structure and the existence
of a pushout as above is required but not specified.
One can recover the open cells by the path-components
of $X_n-X_{n-1}$ and the closed cells by the closure of the open
cells, but not the characteristic maps $Q_i$.
Therefore one has to show that the simple structure on $E_n$ is independent
of the choice of these pushouts. This is done by thickening
$X_{n-1}$ into $X_n$. The details of the argument
are similar to the one given in the proof of
\cite[Lemma 7.13]{Lueck(1989)} and can be found 
in~\cite[Subsection~3.2]{Steimle(2007)}. 

\begin{remark} \label{rem:simple_structure_on_trivial_fibration} If
  $p$ is trivial, i.e., $p\colon B\times F\to B$ is the projection map,
  and $F$ is a finite $CW$-complex, then for any spider $s$, the
  simple structure $\xi(b,s,\xi_{\can}(F))$ on
  $B\times F$ agrees with the product simple structure.
\end{remark}

The dependence of the simple structure on the choice of $(b,s,\zeta)$
is described in the next lemma. Therefore suppose that another choice
$(b',s',\zeta')$ has been made, with $b'\in B$, $s'$ a spider at $b'$,
and $\zeta'$ a simple structure on the fiber $F_{b'}$.

\begin{lemma} \label{lem:dependence_of_simple_structure_on_(b,s,xi(F_b))}
Suppose that $B$ is path-connected. Given a cell $c \in I(B)$, let $u_c$ be any path in
the interior of $c$ from $w_c(1)$ to $w_c'(1)$, where $w_c$ and $w_c'$
are given by the spiders $s$ and $s'$, and let $v_c$ be the path
$w_c \ast u_c \ast (w_c')^-$. Then the homotopy class
relative endpoints $[v_c]$ is independent of $u_c$.
If we denote by $(i_{b'})_*\colon \Wh(\pi(F_{b'})) \to \Wh(\pi(E))$ the homomorphism
induced by the inclusion $i_{b'}\colon F_{b'} \to E$, the following holds in $\Wh(\pi(E))$
\begin{eqnarray*}
\lefteqn{
\tau\big((E,\xi(b,s,\zeta)) \xrightarrow{\id} (E,\xi(b',s',\zeta'))\big)}
& & \\
& \hspace{15mm}  = &
\sum_{c \in I(B)} (-1)^{\dim(c)}\cdot (i_{b'})_*
\tau\big((F_b,\zeta) \xrightarrow{t([v_c])} (F_{b'},\zeta')
\big)
\end{eqnarray*}
\end{lemma}
\begin{proof}
This follows from Lemma~\ref{lem:properties_of_Whitehead_torsion} and
Lemma~\ref{lem:fiber_trivializations}.
\end{proof}

Let $p\colon E \to B$ be a fibration
whose fiber has the homotopy type of a finite $CW$-complex.
We can assign to it a class
\begin{eqnarray}
\Theta(p) & \in & H^1\bigl(B,\Wh(\pi(E))\bigr)
\label{Theta(p)}
\end{eqnarray}
as follows. For simplicity we assume that $B$ is path-connected.
Given $b \in B$, a loop $w$ at $b$ in $B$ and
a simple structure $\zeta$ on $F_b$, we can compute the 
Whitehead torsion of the fiber transport along $w$
\begin{eqnarray*}
(i_b)_*\tau\big(t([w]) \colon (F_b,\zeta)\to (F_b,\zeta)\big) & \in & \Wh(\pi(E))
\end{eqnarray*}
for $i_b\colon F_b \to E$ the inclusion. From
Lemma~\ref{lem:properties_of_Whitehead_torsion} and
Lemma~\ref{lem:fiber_homotopies} one concludes
that this element is independent of the choice of
$\zeta$ and that we obtain a group homomorphism
$\pi_1(B,b) \to \Wh(\pi(E))$. It
defines an element  $\Theta(p) \in H^1\bigl(B;\Wh(\pi(E))\bigr)$
which is independent of the choice of $b \in B$.

\begin{definition} \label{def:simple_fibration}
Let $p\colon E \to B$ be a fibration whose fiber has the 
homotopy type of a finite $CW$-complex. We call $p$ \emph{simple} if $\Theta(p\vert_C) = 0$
holds for any component $C \in \pi_0(B)$ with respect to the restriction
$E\vert_C \to C$.
\end{definition}

\begin{lemma} \label{lem:fiber_bundles_are_simple}
Let $p\colon E \to B$ be a locally trivial fiber bundle with
a finite $CW$-complex as typical fiber and
paracompact base space. Then it is a simple fibration.
\end{lemma}
\begin{proof}
It is a fibration by \cite[page 33]{Whitehead(1978)}. It is
simple, since the fiber transport in such a bundle is given by
homeomorphisms and the Whitehead torsion of a homeomorphism is
trivial (see~\cite{Chapman(1974)}).
\end{proof}

\begin{corollary}
\label{cor:dependence_of_simple_structure_on_(b,s,xi(F_b)_for_simple_p}
Consider the situation of
Lemma~\ref{lem:dependence_of_simple_structure_on_(b,s,xi(F_b))}.
Assume that $p$ is simple. Define
$$\tau_0 := (i_{b'})_*\tau\big(t \colon (F_b,\zeta) \to (F_{b'},\zeta')\big),$$
where $t\colon F_b \to F_{b'}$ is a homotopy equivalence representing
the fiber transport $t([w])$ for some path $w$ from
$b$ to $b'$.

Then $\tau_0$ is independent of the choice of $w$ and
\begin{eqnarray*}
\tau\bigl(\id\colon  (E,\xi(b,s,\zeta)) \to (E,\xi(b',s',\zeta'))\big)
& = &
\chi(B) \cdot \tau_0.
\end{eqnarray*}
\end{corollary}

\begin{notation} \label{not:_xi(E,b,xi(F_b))}
Let $p\colon E \to B$ be a fibration with path-connected finite
$CW$-complex as base space $B$ such the homotopy fiber has the
homotopy type of a finite $CW$-complex. Suppose that $p$ is
simple. Then the simple structure
$\xi(b,s,\zeta)$ of~\eqref{preferred_simple_structure_on_total_space}
is independent of the spider $s$ by
Corollary~\ref{cor:dependence_of_simple_structure_on_(b,s,xi(F_b)_for_simple_p}
and will be denoted briefly by $\xi(b,\zeta)$.
\end{notation}

\begin{corollary} \label{cor:simple_structure_on_E_for_chi(B)_is_0}
Let $p\colon E \to B$ be a fibration
such that $B$ is a path-connected finite $CW$-complex
with $\chi(B) = 0$ and
the fiber has the homotopy type of a finite $CW$-complex.
Suppose that $p$ is simple.
Then $E$ carries a preferred simple structure.
\end{corollary}

The next three lemmas describe the extent of compatibility of our construction
with fiber homotopy equivalences, pushouts, and pullbacks by simple
homotopy equivalences.

\begin{lemma} \label{lem:simple_structure_and_fiber_homotopy_equivalences}
Let $p\colon E \to B$ and $p'\colon E' \to B$ be fibrations and
$(\overline{f},\id)\colon p \to p'$ be a fiber homotopy equivalence.
Let $b \in B$ be a base point and let $s$ be a spider at $b$.
Fix simple structures $\zeta$ and $\zeta'$
on the fibers $F_b$ and $F'_b$ of $p$ and $p'$ over $b$.
Let $\overline{f}_b\colon F_b \to F'_b$ be the homotopy equivalence induced by
$\overline{f}$. From $\overline{f}$ we obtain an isomorphism
$\overline{f}_*\colon H^1\bigl(B,\Wh(\pi(E))\bigr) 
\xrightarrow{\cong} H^1\bigl(B;\Wh(\pi(E'))\bigr)$.

Then we get
\begin{eqnarray*}
\tau\big(\overline{f} \colon (E,\xi(b,s,\zeta))
\to (E',\xi(b,s,\zeta'))\big) & = &
\chi(B) \cdot (i_b)_*\tau
\big(\overline{f}_b\colon (F_b,\zeta) \to (F_b',\zeta')\big);
\\
\Theta(p') & = & \overline{f}_*(\Theta(p)).
\end{eqnarray*}
\end{lemma}
\begin{proof}
This follows from Lemma~\ref{lem:properties_of_Whitehead_torsion} and
Lemma~\ref{lem:fiber_trivializations}.
\end{proof}

For the next lemma, the following extension of the notion of a spider from spaces to maps will be useful: 

\begin{definition} \label{def:spider_for_f} Let $X$ be a $CW$-complex, $f\colon
  X\to B$ a map to a path-connected space, $b\in B$. A \emph{spider} at $b$
  for $f$ is a collection of paths $w_c$ in $B$, indexed by $c \in I(X)$, such
  that $w_c(0) = b$ and $w_c(1)$ is the image under $f$ of a point in the open
  cell $c$.
\end{definition}

\begin{remark}
\label{rem:spiders_for_maps} 
Recall that the choice of a spider is, in general, a necessary ingredient to
construct a simple structure on the total space of a fibration. This
construction can now be generalized: Let $p\colon E\to B$ be a fibration over a
path-connected space $B$, let $X$ be a finite $CW$-complex, and let $f\colon X\to B$ be a
map. Given $b\in B$, a spider $s$ for $f$ at $b$, and a simple structure $\zeta$
on the fiber $F_b$ �of $p$ over $b$, the construction 
of~\eqref{eq:simple_structure_on_total_space} generalizes to the construction of a
simple structure $\xi(f,b,s,\zeta)$ on $f^*E$. (Just use the fiber transport
over $B$ instead of the fiber transport over $X$.) Of course this is the previous 
construction~\eqref{eq:simple_structure_on_total_space} in the
special case $B = X$ and $f = \id_B$.
\end{remark}

This construction will be used in the following situation: Let \xycomsquare{B_0}{\phi}{B_1}{i}{j}{B_2}{\Phi}{B} be a pushout
of finite $CW$-complexes, such that $B$ is connected, $(B_2,B_0)$ and $(B,B_1)$ are
$CW$-pairs, the maps $i$ and $j$ are inclusions, and $B$ is obtained
as a $CW$-complex from $B_1$ by attaching the relative cells of
$(B_2,B_0)$. 

Suppose that $p\colon E\to B$ is a simple fibration. Let $b\in B$, let $\zeta$ be a simple structure on the fiber $F_b$, and denote by $\xi=\xi(b,\zeta)$ the associated simple structure on $E$.

Choose any spider $s_1$ for $j$ at $b$, any spider $s_2$ for $\Phi$ at $b$, and any spider $s_0$ for $j\circ\phi$ at $b$. Thus we obtain a simple structure $\xi_2=\xi(\Phi,b,s_2,\zeta)$ on $\Phi^*E$, and similarly simple structures $\xi_0$ and $\xi_1$ on the corresponding spaces in the following diagram:
\begin{equation}\label{eq:simple_pushout_square}
\xymatrix{
{\bigl(\phi^*(E\vert_{B_1}),\xi_0\bigr)} \ar[r]^{\overline{\phi}} \ar[d]^{\overline{i}} & {\bigl(E\vert_{B_1},\xi_1\bigr)} \ar[d]^{\overline{j}}\\
  {\bigl(\Phi^*E,\xi_2\bigr)}\ar[r]^{\overline{\Phi}} & {(E,\xi)}
}
\end{equation}

\begin{lemma}\label{lem:compatibility_of_simple_structure_with_pushouts}
Under the assumptions above the square \eqref{eq:simple_pushout_square} is a simple pushout, i.e., the pushout simple structure on $E$
  agrees with $\xi$. 
\end{lemma}

\begin{proof}
As $p$ is simple, we can use any spider for $B$ to construct $\xi$. We will make use of the following construction of a spider $s$ for $B$ at $b$ out of $s_1$ and $s_2$: By hypothesis, the set of open cells $I(B)$ of $B$ is a disjoint union 
\[I(B)=I(B_1)\coprod I(B_2,B_0)\]
of the set of open cells of $B_1$ and the set of open cells of the relative $CW$-complex $(B_2,B_0)$. So we can just take the two collections of paths given by $s_1$ and $s_2$ together to obtain a spider $s$ for $B$. More precisely, given an open cell $c$ of $B$, define $s$ by letting $w_c$ be the corresponding path of $s_1$, resp.~$s_2$, if $c\in I(B_1)$, resp.~$c\in I(B_2,B_0)\subset I(B_2)$.

We claim that even if $p$ is non-simple, the square \eqref{eq:simple_pushout_square} is a simple pushout provided we use this particular spider $s$ to construct $\xi=\xi(b,s,\zeta)$.

To prove this claim, we proceed by induction over the dimension of the relative
$CW$-complex $(B_2,B_0)$. If its dimension is $n$, let $B_2^{(n-1)}$ and $B^{(n-1)}$ be the relative $(n-1)$-skeleta. Denote by $i' \colon B_0 \to B_2^{(n-1)}$ and 
$i'' \colon B_2^{(n-1)} \to B_2$ the inclusions, and let $j'\colon B_1\to B^{(n-1)}$ and $j''\colon B^{(n-1)}\to B$ be the corresponding inclusions of subcomplexes of $B$. Denote by $\Phi' \colon
B_2^{(n-1)} \to B^{(n-1)}$ the the restriction of $\Phi$.  We obtain the following commutative
diagram
$$\xymatrix{
\phi^*(E\vert_{B_1}) \ar[r]_{\overline{\phi}} \ar[d]_{\overline{i'}}
& 
E\vert_{B_1} \ar[d]_{\overline{j'}}
\\
(\Phi')^*(E\vert_{B^{(n-1)}}) \ar[r]_{\overline{\Phi'}} \ar[d]_{\overline{i''}}
& 
E\vert_{B^{(n-1)}} \ar[d]_{\overline{j''}}
\\
\Phi^*E \ar[r]_{\overline{\Phi}} & E
}
$$

Notice that, by restricting $s$ and $s_2$, we obtain spiders on $B^{(n-1)}$ and $B_2^{(n-1)}$ for $j''$ and $j''\circ\Phi'$ at $b$ and hence can endow all the spaces of the diagram with simple structures. 

We have to show that for the outer square
the simple structure on $E$ agrees with the pushout simple structure.
One easily checks that it suffices to show that the corresponding
statement holds for the upper and the lower square. This is true for the
upper square by induction hypothesis. For the lower square this is a direct consequence of the 
construction of the simple structure on $E$.
\end{proof}

\begin{lemma}\label{lem:compatibility_of_simple_structure_with_simple_pullback}
  Let $f\colon B' \to B$ be map of finite $CW$-complexes. Let $p\colon
  E \to B$ be a fibration whose homotopy fiber has the homotopy type
  of a finite $CW$-complex. Suppose that $p$ is simple.  Suppose that
  $f\colon B' \to B$ is a simple homotopy equivalence. Let
  \xycomsquare{f^*E}{\overline{f}}{E}{p_f}{p}{B'}{f}{B} be the
  pullback. For every component $C \in \pi_0(B')$ choose a base point
  $x(C)$.  For every $C \in \pi_0(B')$ equip the fiber $(p_f)^{-1}(x(C))$
  with a simple structure $\zeta'_C$ and the fiber
  $p^{-1}\bigl(f(x(C)\bigr)$ with  a simple structure $\zeta_C$
  such that 
  \[
  \tau\bigg(\overline{f}|_{(p_f)^{-1}(x(C))} \colon 
  \bigl((p_f)^{-1}(x(C)),\zeta'_C\bigr)
    \to \bigl(p^{-1}(f(x(C))),\zeta_C\bigr)\biggr) = 0.
  \]

  Equip $f^*E$ and $E$ with the simple structures $\xi'$ and $\xi$ 
  associated to these choices in Notation~\ref{not:_xi(E,b,xi(F_b))}. Then
  $$\tau\big(\overline{f} \colon (f^*E,\xi')\to
  (E,\xi)\big) = 0.$$
  \end{lemma}
  \begin{proof} Recall that a map is a simple homotopy equivalence if it
  is, up to homotopy, a composition of elementary collapses and expansions.  Because of
  Lemma~\ref{lem:properties_of_Whitehead_torsion} we can assume
  without loss of generality that $f\colon B' \to B$ is an elementary
  expansion, i.e., $f$ is the inclusion of a $CW$-subcomplex, where
  $B$ is obtained from $B'$ as a pushout $B=D^{n+1}\cup_{D^n} B'$,
  with an attaching map $D^n\to B'$ mapping into the $n$-skeleton and
  its restriction onto $S^{n-1}$ mapping into the
  $(n-1)$-skeleton. The inclusion of $D^n$ into $D^{n+1}$ is given by
  identifying $D^n$ with the upper hemisphere of $S^n$.

  By Lemma~\ref{lem:compatibility_of_simple_structure_with_pushouts},
  it is enough to show that the inclusion from $E\vert_{D^n}$ to
  $E\vert_{D^{n+1}}$ is simple. The base spaces in these fibrations
  are contractible; hence we can assume by
  Lemma~\ref{lem:simple_structure_and_fiber_homotopy_equivalences}
  that the fibrations are products. In that case, by 
  Remark~\ref{rem:simple_structure_on_trivial_fibration},
  the simple structures
  on the total spaces are the product simple structures, and by the
  product formula the claim follows from the fact that the inclusion
  $D^n\to D^{n+1}$ is simple.
\end{proof}

\begin{remark} \label{rem:simple_structures_on_B_and_F-yield_one_on_E}
Let $p\colon E\to B$ be a simple fibration over a path-connected base
space $B$ carrying a simple structure $\eta$, and suppose we are
given a simple structure $\zeta$ of the fiber over a point $b$.
Notice that
Lemma~\ref{lem:compatibility_of_simple_structure_with_simple_pullback}
gives us the possibility to define a simple structure on the total
space $E$: Choose a finite $CW$-model $f\colon X\to B$
representing $\eta$ and consider the pullback structure map
$\overline{f}\colon f^*E\to E$. We can arrange by
possibly changing $f$ up to homotopy that 
$b=f(x)$ for some $x\in X$. Then $f^*E$ carries the simple
structure $\xi(x,\zeta)$; give $E$ the simple structure for
which the torsion of $\overline{f}$ vanishes. We are going to
denote this simple structure by $\xi(\eta,b,\zeta)$.
\end{remark}

Let $M$ be a closed topological manifold. Then, by
Kirby-Siebenmann~\cite[Essay~III, Theorem~4.1 on page~118]{Kirby-Siebenmann(1977)},
there is a preferred simple structure
\begin{eqnarray}
\xi^{\topo}(M) \mbox{ on } M,
\label{preferred_simple_structure_on_a_topological_manifold}
\end{eqnarray}
which is defined by considering any triangulated closed disc bundle
over $M$: The simple structure on the disc bundle obtained from the
triangulation induces the preferred simple structure on $M$ via the
retraction onto $M$. This simple structure agrees with the one
obtained by any triangulation or by any handlebody decomposition (more generally
what they call TOP s-decomposition) of $M$, whenever they exist
(see~\cite[Essay~III, Theorem~5.10 on page~131 and 
Theorem~5.11 on page~132]{Kirby-Siebenmann(1977)}).

\begin{lemma} \label{lem:simple_structures_on_bundles}
  Let $F \to M \to B$ be a locally trivial bundle of closed
  topological manifolds with path connected $B$.  Then we get:
\begin{eqnarray*}
\Theta(p) & = & 0;
\\
\xi^{\topo}(M) & = & \xi\bigl(\xi^{\topo}(B),b,\xi^{\topo}(F)),
\end{eqnarray*}
where $\xi\bigl(\xi^{\topo}(B),b,\xi^{\topo}(F))$ 
has been defined in Remark~\ref{rem:simple_structures_on_B_and_F-yield_one_on_E}.
\end{lemma}
\begin{proof} We have already proved $\Theta(p) = 0$ in
  Lemma~\ref{lem:fiber_bundles_are_simple}. Moreover, if the bundle
  happens to be globally trivial, then the simple structure
  $\xi(M,b,\xi^{\topo}(F))$ agrees with $\xi^{\topo}(B\times F)$ by 
  Remark~\ref{rem:simple_structure_on_trivial_fibration}.

  Consider now the general case. We need not take care of the base
  point $b$, as the torsion of every fiber transport is zero (see
  Corollary~\ref{cor:dependence_of_simple_structure_on_(b,s,xi(F_b)_for_simple_p}).
  First suppose that $\dim(B)\geq 6$. Then there exists a
  handlebody decomposition
\[D^n=B_0\subset B_1 \subset\dots\subset B_k=B\]
\cite[Essay~III, \S 2]{Kirby-Siebenmann(1977)}, 
and proceed by induction over $k$. If $k=0$, then the bundle is trivial, and the
claim follows. For the induction step, consider the pushout which
attaches to $B_k$ a handle $H$ to get $B_{k+1}$. By
Lemma~\ref{lem:compatibility_of_simple_structure_with_pushouts}, the
pushout \xycomsquare{\bigl(E\vert_{H\cap B_k}, \xi(\xi^{\topo}(H\cap
  B_k),b,\xi^{\topo}(F)) \bigr)}{}{\bigl(E\vert_{B_k},
  \xi(\xi^{\topo}(B_k),b,\xi^{\topo}(F)) \bigr)}{}{}{\bigl(E\vert_{H},
  \xi(\xi^{\topo}(H),b,\xi^{\topo}(F))
  \bigr)}{}{\bigl(E\vert_{B_{k+1}},\xi(\xi^{\topo}(B_{k+1}),b,\xi^{\topo}(F))
  \bigr)} is also simple. Here, the simple structures of the left
column agree with the preferred structures as the bundles are trivial;
by induction hypothesis, the same is true for the upper right space.
Now the above pushout is one of the equivalent methods in
\cite[Essay~III, \S 5]{Kirby-Siebenmann(1977)} to give
$E\vert_{B_{k+1}}$ its preferred simple structure
$\xi^{\topo}(E\vert_{B_{k+1}})$. Hence the two structures on
$E\vert_{B_{k+1}}$ agree.

We still have to treat the case where $\dim(B)\leq 5$. Take a
1-connected closed topological manifold $N$ with $\dim(N)\geq 6$ and
$\chi(N)=1$, e.g., $(\IC\IP^2 \times \IC\IP^2) \sharp 4 (S^3 \times S^5)$.
Now apply what we have already proved to the fiber bundle
$M\times N\to B\times N$ which is the product of the original bundle
by the identity map on $N$.
This leads to the equality
\begin{equation}\label{eq:simple_structure_on_fibration_times_N}
\xi^{\topo}(M\times N) = \xi\bigl(\xi^{\topo}(B\times N), (b,n),\xi^{\topo}(F)\bigr)
\end{equation}
for any $n\in N$. 
It is not hard to check that the right hand side 
of~\eqref{eq:simple_structure_on_fibration_times_N} coincides with 
$\xi\left(\xi^{\topo}(B),b,\xi^{\topo}(F)\right)\times\xi^{\topo}(N)$. 
Since $\dim(N \times B) \ge 6$, we get
\begin{align*}
  &\tau\bigl(\id \colon (M,\xi^{\topo}(M))\to (M, \xi(\xi^{\topo}(B),b,\xi^{\topo}(F)))\bigr)\\
=\;  &\tau\bigl(\id \colon (M,\xi^{\topo}(M))\to (M, \xi(\xi^{\topo}(B),b,\xi^{\topo}(F)))\bigr) 
\cdot \chi(N) \\
=\; &\tau\bigl(\id \colon (M \times N,\xi^{\topo}(M) \times \xi^{\topo}(N))
\to (M \times N, \xi(\xi^{\topo}(B),b,\xi^{\topo}(F))) \times \xi^{\topo}(N)\bigr)\\
  = \; & \tau\left(\id \colon (M\times N,\xi^{\topo}(M\times N))\to
(M\times N, \xi(\xi^{\topo}(B\times N), (b,n),\xi^{\topo}(F))\right)\\
  = \; & 0
\end{align*}
by Lemma~\ref{lem:properties_of_Whitehead_torsion}~%
\ref{lem:properties_of_Whitehead_torsion:product_formula}.
\end{proof}

\typeout{-----------------------  Section 4 ------------------------}

\section{Turning a map into a fibration}
\label{sec:Turning_a_map_into_a_fibration}

Let $f\colon X \to B$ be a map. Let $\FIB(f)$ be the subspace of
$X \times \map([0,1],B)$ consisting of pairs $(x,w)$ which satisfy
$w(0) = f(x)$. Let $\widehat{f}\colon \FIB(f) \to B$ be the map sending
$(x,w)$ to $w(1)$. Let $\lambda_f\colon X \to \FIB(f)$ be the map which sends
$x \in X$ to $(x,c_{f(x)})$ for $c_{f(x)}$ the constant path at $f(x)$ in $B$.
Denote by $\mu_f\colon \FIB(f) \to X$ the map $(x,w) \mapsto x$.
Then $\widehat{f}\colon \FIB(f) \to B$ is a
fibration, $\lambda_f$ is a homotopy equivalence and
$\widehat{f} \circ \lambda_f = f$, $\mu_f \circ \lambda_f = \id$,
$f \circ \mu_f \simeq \widehat{f}$ and $\lambda_f \circ \mu_f \simeq \id$
\cite[Theorem 7.30 on page 42]{Whitehead(1978)}.
The fiber of $\widehat{f}\colon \FIB(f) \to B$ over $b$
is called the homotopy fiber of $f$ over $b$ and denoted by $\hofib(f)_b$.

\begin{lemma} \label{lem:turning_a_map_into_a_fibration}
\begin{enumerate}

\item  \label{lem:turning_a_map_into_a_fibration:map_already_fibration}
If $f\colon E \to B$ is already a fibration, then $\lambda_f\colon E \to \FIB(f)$ is
a fiber homotopy equivalence over $B$;

\item  \label{lem:turning_a_map_into_a_fibration:homotopic_maps}
If $H\colon X \times [0,1]$ is a homotopy, $f \simeq g\colon X \to B$,
then it induces a fiber homotopy equivalence
$\widehat{H}\colon \FIB(f) \to \FIB(g)$.

\end{enumerate}
\end{lemma}
\begin{proof}\ref{lem:turning_a_map_into_a_fibration:map_already_fibration}
see \cite[Theorem 7.31 on page 43]{Whitehead(1978)}.
\\[1mm]\ref{lem:turning_a_map_into_a_fibration:homotopic_maps}
$\widehat{H}$ sends $(x,w) \in \FIB(f)$ to $(x,v)\in \FIB(g)$ for the following
path $v\colon [0,1] \to B$
$$v(t)=\left\{\begin{array}{ll}
H(x,1-2t) & 0 \le t \le 1/2;
\\
w(2t-1) & 1/2 \le t \le 1.
\end{array}\right.$$
\end{proof}

\typeout{-----------------------  Section 5 ------------------------}

\section{Fiber torsion obstructions}
\label{sec:Fiber_torsion_obstructions}

\begin{definition}[Fiber torsion obstructions] \label{def:fiber_torsion_obstruction}
Let $f\colon M \to B$ be a map of closed topological manifolds
for path-connected $B$. Suppose that
for some (and hence all) $b \in B$ the homotopy fiber
$\hofib(f)_b$ has the homotopy type of a finite $CW$-complex.

\begin{enumerate}

\item Define the element
$$\Theta(f) \in H^1\bigl(B;\Wh(\pi(M))\bigr)$$
to be the image of $\Theta(\widehat{f})$ under the isomorphism
$H^1\bigl(B;\Wh(\pi(\FIB(f)))\bigr) \to H^1\bigl(B;\Wh(\pi(M))\bigr)$ induced
by the homotopy equivalence $\mu_f \colon \FIB(f) \to M$;

\item
Suppose that $\Theta(f)$ vanishes. Let
$(\mu_f \circ i_b)_*\colon \Wh\bigl(\pi(\hofib(f)_b)\bigr) \to \Wh(\pi(M))$ be the
map induced by the composite
$\hofib(f)_b \xrightarrow{i_b} \FIB(f) \xrightarrow{\mu_f} M$.
Define the \emph{fiber torsion obstruction}
$$\hspace{10mm} \tau_{\fib}(f) \; \in \;
\cok\bigl(\chi(B) \cdot (\mu_f \circ i_b)_*\colon
\Wh(\pi(\hofib(f)_b)) \to \Wh(\pi(M))\bigr)
$$
to be class for which a representative in $\Wh(\pi(M))$ is the image of the Whitehead torsion
$$\tau\big(\lambda_f \colon (M,\xi^{\topo}(M)) \to
(\FIB(f),\xi(b,\zeta))\big)$$
under the isomorphism
$(\mu_f)_*\colon \Wh\bigl(\pi(\FIB(f))\bigr) \to \Wh(\pi(M))$ for some choice of
a base point $b \in B$ and a simple
structure $\zeta$ on $\hofib(f)_b$.
\end{enumerate}
\end{definition}

\begin{remark}[Independence of the base point $b \in B$ and the simple structure
  on $\zeta$ on $\hofib(f)_b)$] Notice that the image of
  $\Wh\bigl(\pi(\hofib(f)_b)\bigr) \to \Wh(\pi(M))$ is independent of the choice
  of $b \in B$.  Namely, let $b'$ be another base point. The fiber transport
  along some path $w$ from $b$ to $b'$ defines a homotopy equivalence $t_w
  \colon \hofib(f)_b \to \hofib(f)_{b'}$ such that the composites $\mu_f \circ
  i_b$ and $\mu_f \circ i_{b'} \circ t_w$ are homotopic and hence induce the
  same map on Whitehead groups.  Hence $$\cok\bigl(\chi(B) \cdot (\mu_f \circ
  i_b)_*\colon \Wh(\pi(\hofib(f)_b)) \to \Wh(\pi(M))\bigr)$$ is independent of
  the choice of $b \in B$.

  Next we show that $\tau_{\fib}(f)$ is a well-defined invariant of $f$ if
  $\Theta(f)$ vanishes.  We have to show that the choice of a base point $b \in
  B$ and of a simple structure $\zeta$ on $\hofib(f)_b$ does not matter.  (We already know
   that the choice of spiders does not play a role.)  Suppose we have made
  a different choice of a base point $b' \in B$ and of a simple structure of
  $\zeta'$ on $\hofib(f)_{b'}$. Then we get from
  Lemma~\ref{lem:properties_of_Whitehead_torsion}~\ref{lem:properties_of_Whitehead_torsion:composition_formula}
  \begin{multline*}
    \tau\big(\lambda_f \colon (M,\xi^{\topo}(M)) \to
    (\FIB(f),\xi(b',\zeta'))\big) \\ - \tau\big(\lambda_f \colon
    (M,\xi^{\topo}(M)) \to (\FIB(f),\xi(b,\zeta))\big)
    \\
    = \tau\bigl(\id\colon (\FIB(f),\xi(b,\zeta)) \to
    (\FIB(f),\xi(b',\zeta'))\bigr).
  \end{multline*}
  Now apply
  Corollary~\ref{cor:dependence_of_simple_structure_on_(b,s,xi(F_b)_for_simple_p}.
\end{remark}

\begin{theorem}\label{the:prop._of_fiber_torsion_obstruction}
Let $f\colon M \to B$ be a map of closed topological manifolds
for path-connected $B$. Then

\begin{enumerate}

\item \label{the:prop._of_fiber_torsion_obstruction:homotopy_invariance}
The element $\Theta(f)$ depends only on the homotopy class of $f$. 
If $\Theta(f)$ vanishes, then the same statement holds for
the fiber torsion obstruction $\tau_{\fib}(f)$;

\item \label{the:prop._of_fiber_torsion_obstruction:obstruction}
If $f$ is homotopic to a map $g\colon M \to B$ which is the projection
of a locally trivial fiber bundle with a closed topological
manifold as fiber, then both $\Theta(f)$ and $\tau_{\fib}(f)$ vanish.
\end{enumerate}
\end{theorem}
\begin{proof}~\ref{the:prop._of_fiber_torsion_obstruction:homotopy_invariance}
  Let $H \colon M \times [0,1] \to B$ be a homotopy between $f$ and $g$. Then
  the fiber homotopy equivalence $\widehat{H} \colon \FIB(f) \to \FIB(g)$ over
  $\id_B$ constructed in the proof of
  Lemma~\ref{lem:turning_a_map_into_a_fibration}~\ref{lem:turning_a_map_into_a_fibration:homotopic_maps}
  has the property that the composite $M \xrightarrow{\lambda_f} \FIB(f)
  \xrightarrow{\widehat{H}} \FIB(g)$ is homotopic to $\lambda_g \colon M \to
  \FIB(g)$. Lemma~\ref{lem:simple_structure_and_fiber_homotopy_equivalences}
  implies that the isomorphism
  $$H^1\bigl(\id_B;\Wh(\pi(\widehat{H}))\bigr) \colon
  H^1\bigl(B;\Wh(\pi(\FIB(f)))\bigr) \to
  H^1\bigl(B;\Wh(\pi(\FIB(g)))\bigr)$$
  sends $\theta(\widehat{f})$  to  $\theta(\widehat{g})$. Hence
  $\theta(f) = \theta(g)$ since homotopic maps induce the same map on Whitehead
  groups.

  Now suppose $\theta(f) = 0$. Hence the fibrations $\widehat{f} \colon
  \FIB(f) \to B$ and $\widehat{g} \colon \FIB(g) \to B$ are simple. Fix a base
  point $b \in B$ and a simple structure $\zeta$ on $\FIB(f)_b$. Equip
  $\FIB(g)_b$ with the simple structure $\zeta'$ for which the homotopy
  equivalence $\FIB(f)_b \to \FIB(g)_b$ induced by $\widehat{H}$ is simple.  Let
  $s$ be any spider for $B$ at $b$.  We conclude from
  Lemma~\ref{lem:simple_structure_and_fiber_homotopy_equivalences} that
  $\tau\bigl(\widehat{H} \colon (\FIB(f),\xi(b,s,\zeta)) \to
  (\FIB(g),\xi(b,s,\zeta'))\bigr)$ vanishes.  This implies $\tau_{\fib}(f) =
  \tau_{\fib}(g)$ since homotopic maps induce the same map on Whitehead groups
  and we have already shown that the definition of $\tau_{\fib}$ is
  independent of the choice of the
  base point $b$, the spider $s$ and the simple structure $\zeta$.
  \\[1mm]~\ref{the:prop._of_fiber_torsion_obstruction:obstruction} 
  Let $g$ be a fiber bundle homotopic to $f$. By assertion~\ref{the:prop._of_fiber_torsion_obstruction:homotopy_invariance}, $\Theta(f)=\Theta(g)$. The fibration associated to $g$ is, by Lemma \ref{lem:turning_a_map_into_a_fibration} \ref{lem:turning_a_map_into_a_fibration:map_already_fibration}, fiber homotopy equivalent to $g$, so Lemma \ref{lem:simple_structure_and_fiber_homotopy_equivalences} allows us to compute $\Theta(g)$ directly from the bundle $g$. Lemma \ref{lem:simple_structures_on_bundles} implies that $\Theta(g)=0$. Now the same arguments show that $\tau_{\fib}(f)=\tau_{\fib}(g)$, and $\tau_{\fib}(g)=0$. 
\end{proof}

\begin{remark} \label{rem:special_case_chi(B)_is_0} Let $f\colon M \to B$ be a
  map of closed topological manifolds for path-connected $B$. If $\chi(B)$
  happens to be zero and $\Theta(f)$ vanishes, the invariant defined in
  Definition~\ref{def:fiber_torsion_obstruction} lives in
  \begin{eqnarray}
    \tau_{\fib}(f) & \in & \Wh(\pi(M)).
    \label{tau(f)_for_chi(B)_is_0}
  \end{eqnarray}
  In other words, if $\chi(B) = 0$, then $\FIB(f)$ carries a preferred simple
  structure $\xi$ by Corollary~\ref{cor:simple_structure_on_E_for_chi(B)_is_0}
  and the element $\tau_{\fib}(f)$ is the image of the Whitehead torsion of the
  map $\lambda\colon (M,\xi^{\topo}(M)) \to (\FIB(f),\xi)$ under the isomorphism
  $(\mu_f)_*\colon \Wh(\pi(\FIB(f)))$ $ \to \Wh(\pi(M))$.
\end{remark}

\begin{example}
  Let $f\colon M \to B$ be a map of closed topological manifolds for
  path-connected $B$ and $M$. Suppose that for some (and hence all) $b \in B$
  the homotopy fiber $\hofib(f)_b$ has the homotopy type of a finite
  $CW$-complex.  Suppose that the Whitehead group of the kernel of $\pi_1(f)
  \colon \pi_1(M) \to \pi_1(B)$ is trivial. This is the case if $\pi_1(f)$ is
  bijective. Then $\Theta(f)$ vanishes.

  This follows from the long exact homotopy sequence of $\FIB(f) \to B$ which
  implies that under the conditions above the map
  $\Wh\bigl(\pi(\hofib(f)_b)\bigr) \to \Wh(\pi(M))$ is trivial.
\end{example}

\typeout{-----------------------  Section 6 ------------------------}

\section{Base space $S^1$}
\label{sec:base_space_S1}

In this section we consider the case, where the base space is the
one-dimensional sphere $S^1$, i.e., we consider a map 
$$f \colon M \to S^1$$
from a connected closed manifold $M$ to $S^1$ whose homotopy fiber has the
homotopy type of a finite $CW$-complex.  In this special
situation we can find a single obstruction $\tau_{\fib}'(f)$ which
carries the same information as our previous invariants $\Theta(f)$
and $\tau_{\fib}(f)$ and has a nice description in terms of mapping
tori. $\tau_{\fib}(f)$ agrees with the obstruction $\tau(f)$ defined 
in \cite{Farrell(1971)}.

We begin with the definition of $\tau_{\fib}'(f)$.  
Let $e \colon \IR \to S^1, \quad t \mapsto \exp(2\pi i t)$ 
be the universal covering of $S^1$. We abbreviate the homotopy 
fiber over $e(0)$ by
$F:=\hofib(f)_{e(0)} = \FIB(f)_{e(0)}$.

Equip $S^1$ with the $CW$-structure whose $0$-skeleton is $e(0)$ and
whose $1$-skeleton is $S^1$. Let $s$ be the spider based at $e(0)$
which is given by the constant path at $e(0)$ for the $0$-cell and by
the path $w \colon [0,1] \to S^1$ sending $t$ to $\exp(\pi i t)$ for
the $1$-cell.  Equip $\FIB(f)$ with the simple structure
$\xi\bigl(e(0),s,\zeta\bigr)$ defined
in~\eqref{preferred_simple_structure_on_total_space} for any choice of
simple structure $\zeta$ on $F$. Because of
Lemma~\ref{lem:dependence_of_simple_structure_on_(b,s,xi(F_b))} the
simple structure $\xi\bigl(e(0),s,\zeta\bigr)$ is independent of the
choice of $\zeta$ and we will write $\xi\bigl(e(0),s\bigr)$. Then
\begin{eqnarray}
  \tau_{\fib}'(f) & \in & \Wh(\pi(M))
  \label{tauprime(f)}
\end{eqnarray}
is defined to be the Whitehead torsion of the canonical homotopy
equivalence $\mu_f \colon \FIB(f) \to M$ with respect to the simple
structure $\xi\bigl(e(0),s\bigr)$ on $\FIB(f)$ and the simple
structure associated to the structure $\xi^{\topo}(M)$ of a 
closed topological manifold on $M$.

In the sequel we identify $H^1(S^1;\Wh(\pi(M)) = \Wh(\pi(M))$
using the standard generator of $\pi_1(S^1) \cong  H_1(S^1)\cong \IZ$ represented by the
identity map $\id \colon S^1 \to S^1$. In
particular $\Theta(f)$ becomes an element  in $\Wh(\pi(M))$.

Complex conjugation defines an orientation reversing self-diffeomorphism
$$\con \colon S^1 \to S^1, \quad z \mapsto \overline{z}.$$

\begin{lemma}
\label{lem_thetaprime_compared_with_Theta_and_tau}
\
\begin{enumerate} 
\item \label{lem_thetaprime_compared_with_Theta_and_tau:theta_from_tauprime}
We have 
$$ \Theta(f) = \tau_{\fib}'(f) - \tau_{\fib}'(\con \circ f);$$
\item \label{lem_thetaprime_compared_with_Theta_and_tau:tau_from_tauprime}
If $\Theta(f) = 0$, then
$$\tau_{\fib}(f) = \tau_{\fib}'(f);$$
\item \label{lem_thetaprime_compared_with_Theta_and_tau:vanishing}
We have $\tau_{\fib}'(f) = 0$ if  $\Theta(f) = 0$ and $\tau_{\fib}(f)=0$ hold.
\end{enumerate}
\end{lemma}
\begin{proof}~\ref{lem_thetaprime_compared_with_Theta_and_tau:theta_from_tauprime}
  Let $\overline{s}$ be the spider on $S^1$ with base point $e(0)$ which is
  given by the constant path at $e(0)$ for the $0$-cell and 
  by the path $\overline{w} \colon [0,1] \to S^1$ sending $t$ to
  $\overline{\exp(\pi i t)} = \exp(-\pi i t)$ for the $1$-cell.  
  Obviously $\tau_{\fib}'(\con \circ f)$ is 
  the Whitehead torsion of the canonical homotopy
  equivalence $\mu_f \colon \FIB(f) \to M$ with respect to the simple
  structure $\xi\bigl(e(0),\overline{s}\bigr)$ on $\FIB(f)$ and the
  simple structure associated to the structure of a closed manifold on
  $M$.  Hence we conclude from
  Lemma~\ref{lem:properties_of_Whitehead_torsion}~%
\ref{lem:properties_of_Whitehead_torsion:composition_formula}
$$\tau_{\fib}'(\con \circ f) - \tau_{\fib}'(f) = 
\tau\bigl(\id \colon (\FIB(f),\xi(e(0),\overline{s})) \to 
(\FIB(f),\xi(e(0),s))\bigr).$$
Now the claim follows from Lemma~\ref{lem:dependence_of_simple_structure_on_(b,s,xi(F_b))} 
and the definition of $\Theta(f)$.
\\[1mm]~%
\ref{lem_thetaprime_compared_with_Theta_and_tau:tau_from_tauprime}
This follows from the definitions.
\\[1mm]~%
\ref{lem_thetaprime_compared_with_Theta_and_tau:vanishing}
This follows from 
assertions~\ref{lem_thetaprime_compared_with_Theta_and_tau:theta_from_tauprime}
and~\ref{lem_thetaprime_compared_with_Theta_and_tau:tau_from_tauprime}.
\end{proof}

\begin{remark}[Mapping tori]
  \label{rem:mapping_tori}
  Given a self-map $v \colon Y \to Y$, define its \emph{mapping torus} $T_v$ by
  the pushout
  \xycomsquare{Y\coprod Y}{\id\coprod\id}{Y}{}{}{\cyl(v)}{}{T_v} where the left
  vertical arrow is the inclusion of the front and the back into the mapping
  cylinder. This corresponds to identifying in $Y \times [0,1]$ the point
  $(y,0)$ with $(v(y),1)$ for all $y \in Y$.

  If $Y$ has the homotopy type of a finite $CW$-complex, we can choose a simple
  structure on $Y$ and equip $\cyl(v)$ with the simple structure such that the
  back inclusion is a simple homotopy equivalence. Equip the mapping torus $T_v$
  with the pushout simple structure (see
  Section~\ref{sec:Simple_structures_and_Whitehead_torsion}).  This simple
  structure is independent of the choice of the simple structure on $Y$ by
  Lemma~\ref{lem:properties_of_Whitehead_torsion}. Hence a mapping torus of a
  self-map of a space with the homotopy type of a finite $CW$-complex has a
  preferred simple structure which we will use in the sequel without any further
  notice.

  Let $Y_1$ and $Y_2$ be homotopy equivalent spaces of the homotopy type of a
  finite $CW$ complex. Consider self-homotopy equivalences $v_i \colon Y_i \to
  Y_i$ for $i = 1,2$ such that $v_2 \circ u \simeq u \circ v_1$ for some
  homotopy equivalence $u\colon Y_1\to Y_2$. Choose a homotopy $h \colon v_2
  \circ u \simeq u \circ v_1$.  Then $h$ induces maps $\cyl(v_1)\to\cyl(v_2)$
  and $T_{u,h}  \colon T_{v_1}\to T_{v_2}$.
  The homotopy class of the latter map depends on the choice of $u$ and $h$, but
  not its Whitehead torsion.  Namely,
  Lemma~\ref{lem:properties_of_Whitehead_torsion} implies
$$\tau\bigl(T_{u,h} \colon T_{v_1} \to T_{v_2} \bigr) = 0.$$
\end{remark}

Consider the pullback
\xycomsquare{\overline{M}}{\overline{e}}{M}{\overline{f}}{f}{\IR}{e}{S^1}
of $f$ with the universal covering $e$.
Consider the map $l_1\colon \overline{M}\to\overline{M}$, induced by the action of
$1\in\IZ\cong\pi_1(S^1)$ by deck transformations. Since
$\overline{e} \circ l_1 = \overline{e}$, the map
$\overline{M} \times [0,1] \to M$ sending $(x,t)$ to $\overline{e}(x)$
induces a homotopy equivalence
$$\widehat{e} \colon T_{l_1} \to M.$$

\begin{lemma} \label{lem:tauprime(f)_and_T_l_1}
We get
\begin{eqnarray*}
\Theta(f) 
& = & 
\overline{e}_*\bigl(\tau(l_1 \colon \overline{M} \to \overline{M})\bigr);
\\
\tau'_{\fib}(f) 
& = &
\tau\bigl(\widehat{e} \colon T_{l_1} \to M\bigr),
\end{eqnarray*}
where we use the preferred simple structures on the mapping torus
$T_{l_1}$ and on the closed manifold $M$, and any simple structure on
$\overline{M}$.
\end{lemma}
\begin{proof}
There is an explicit homotopy equivalence
$$
h\colon \overline{M}  \xrightarrow{\simeq}  F,
$$
which sends $x \in \overline{M}$ 
to  $(\overline{e}(x),w)\in F$ for the path 
$w(t)=\exp\bigl(2\pi i \overline{f}(x)(1-t)\bigr)$.  
Let $t \colon F \to F$ be given by the canonical fiber transport
along the standard generator of $S^1$. It sends a pair $(x,w) \in F$ 
to the pair $(x,w \ast v)$ for the path $v = e|_{[0,1]}$.
We have by definition
$$\Theta(f) = (\mu_f \circ i)_*\bigl(\tau(t\colon F \to F)\bigr)$$
for any choice of simple structure on $F$, where $i \colon F \to \FIB(f)$ is the 
inclusion and $\mu_F \colon \FIB(f) \to M$ is the canonical map.
We have $h \circ l_1 = t \circ h$.  
Lemma~\ref{lem:properties_of_Whitehead_torsion} 
implies that for any choice of simple structure on $\overline{M}$
$$\Theta(f) = 
(\mu_f \circ i \circ h)_*\bigl(\tau(l_1\colon \overline{M} \to \overline{M})\bigr).$$
Since $\overline{e} = \mu_f \circ i \circ h$, we conclude
$$\Theta(f) 
= \overline{e}_*\bigl(\tau(l_1 \colon \overline{M} \to \overline{M})\bigr).$$
Define
$$\alpha' \colon F \times [0,1] \to \FIB(f), \quad \bigl((x,w),s\bigr) \mapsto (x,w_s),$$
where $w_s$ is the path sending $s' \in [0,1]$ to
$$w_s(s') := \begin{cases}
w\bigl((s+1)s'\bigr) & 0 \le s' \le \frac{1}{s+1};
\\
\exp\bigl(2\pi i(s'(s+1) -1)\bigr) & \frac{1}{s+1} \le s' \le 1.
\end{cases}
$$
The following diagram commutes 
\xycomsquare{F \times[0,1]}{\alpha'}{\FIB(f)}{\pr}{f}{[0,1]}{e}{S^1} 
The map $\alpha'$ induces over $0$ the identity $F \to F$ and over $1$ the map
$t\colon F\to F$. Hence the map
$\alpha'$ induces an explicit homotopy equivalence
$$\alpha \colon  T_t \xrightarrow{\simeq} \FIB(f).$$
By definition
$$\tau'_{\fib}(f) = \tau\bigl(\mu_f \circ \alpha \colon  T_t \to M\bigr).$$
Since $h \circ l_1 = t \circ h$, the map $h$ induces an explicit
homotopy equivalence
$$\beta  \colon T_{l_1} \to T_t.$$
We conclude from Remark~\ref{rem:mapping_tori}
$$\tau'_{\fib}(f) 
= \tau\bigl(\mu_f \circ \alpha \circ \beta \colon T_{l_1} \to M\bigr).$$
Since $\widehat{e} = \mu_f \circ \alpha \circ \beta$, we conclude
\[\tau'_{\fib}(f) = \tau\bigl(\widehat{e} \colon T_{l_1} \to M\bigr).\qedhere\]
\end{proof}

\typeout{-----------------------  Section 7 ------------------------}

\section{Gluing $h$-cobordisms}
\label{sec:Gluing_h-cobordisms}

In this section we consider the illuminating example of a map
$M \to S^1$ which is obtained from gluing the two ends of an $h$-cobordism
together.

Let $(W,\partial_0W,\partial_1W)$ be a topological $h$-cobordism,
i.e., a closed manifold $W$ whose boundary is the
disjoint union $\partial W = \partial_0 W \coprod \partial_1W$
such that both inclusions $i_k\colon \partial_kW \to W$ are
homotopy equivalences. Its Whitehead torsion
\begin{eqnarray}
\tau(W) & \in & \Wh(\partial_0 W)
\label{Whitehead_torsion_of_h-cobordism}
\end{eqnarray}
is defined by the image of the Whitehead
torsion $\tau(i_0\colon \partial_0 W \to W)$ under the isomorphism
$\Wh(\pi(W)) \xrightarrow{\cong} \Wh(\pi(\partial_0W))$ induced by
$i_0^{-1}$. Let $g\colon \partial_1W \to \partial_0W$ be a
homeomorphism. Let $W_g$ be the closed topological manifold obtained from $W$
by gluing $\partial_1W$ to $\partial_0W$ by $g$. 
Choose any continuous map $f'\colon W \to [0,1]$ with
$f'(\partial_0W) = \{0\}$ and $f'(\partial_1W) = \{1\}$. Let
$$f_g\colon W_g \to S^1$$ 
be the map induced by $f'$. Since $[0,1]$ is convex, $f_g$ is unique up to homotopy.
Let $l \colon \partial_0W \to W_g$ be the obvious inclusion.
Let $w_1 \colon \pi_1(\partial_0W) \to \{\pm 1 \}$ be the orientation homomorphism
of $\partial_0W$.  The $w_1$-twisted anti-involution on the group ring
$\IZ \bigl[\pi_1(\partial_0W)\bigr]$ is given by
\begin{eqnarray*}
\overline{\sum_{g \in \pi_1(\partial_0W)} \lambda_g \cdot g} 
& = & 
\sum_{g \in \pi_1(\partial_0W)} w_1(g) \cdot  \lambda_g \cdot g^{-1}.
\label{w_1-twisted_involution}
\end{eqnarray*}
Let 
$$\ast\colon \Wh(\pi(\partial_0W)) \to \Wh(\pi(\partial_0W))$$ be the induced
involution. It corresponds geometrically
to turning an $h$-cobordism upside down \cite[\S 10]{Milnor(1966)}.
Namely, if $W^*$ is the $h$-cobordism with the roles of 
$\partial_0 W$ and $\partial_1 W$ interchanged, we get
(see \cite[Duality Theorem on page~394]{Milnor(1966)})
$$\tau(W^*) =  
(-1)^{\dim(\partial_0 W)} \cdot (i_1)_*^{-1} \circ (i_0)_* \circ \ast(\tau(W)).$$
\begin{lemma}\label{lem_theta_and_tau_fib_for_h_cobordism}
Let 
$$\ast \colon \Wh(\pi(W_g)) \to \Wh(\pi(W_g))$$ 
be the $w_1(W_g)$-twisted involution,
where $w_1(W_g) \colon \pi_1(W_g) \to \{\pm 1\}$ is the orientation homomorphism of $W_g$.
Then:
\begin{enumerate}
\item \label{lem_theta_and_tau_fib_for_h_cobordism:Theta}
We have
$$\Theta(f_g) 
= l_*\bigl(\tau(g \circ i_1^{-1} \circ i_0)\bigr)
= \bigl((-1)^{\dim(W)} \cdot  \ast + \id\bigr)\circ l_*(\tau(W));$$

\item \label{lem_theta_and_tau_fib_for_h_cobordism:tauprime}
We have 
$$\tau_{\fib}'(f_g) = (-1)^{\dim(W)} \cdot \ast \circ l_*(\tau(W))
= l_*(\tau(W)) - \Theta(f_g);$$
\item \label{lem_theta_and_tau_fib_for_h_cobordism:tau}
If $\Theta(f_g) = 0$, then
$$\tau_{\fib}(f_g) = -l_*(\tau(W)).$$
\item \label{lem_theta_and_tau_fib_for_h_cobordism:equality_of_vanishing}
The following assertions are equivalent:
\begin{enumerate}
\item $l_*(\tau(W)) = 0$;
\item $\tau_{\fib}'(f_g) = 0;$
\item $\Theta(f_g) = 0$ and $\tau_{\fib}(f_g) = 0$.
\end{enumerate}
\end{enumerate}
\end{lemma}

\begin{proof}~\ref{lem_theta_and_tau_fib_for_h_cobordism:Theta}
Since $\partial_0W$ is part of the boundary of $W$, we get
$w_1(W_g) \circ \pi_1(l) = w_1(\partial_0W)$. 
Hence $l_* \colon \Wh(\pi(\partial_0 W)) \to \Wh(\pi(W_g))$ is compatible with the involutions.

Consider the pullback
\xycomsquare{\overline{W_g}}{\overline{e}}{W_g}{\overline{f_g}}{f_g}{\IR}{e}{S^1}
of $f_g$ with the universal covering $e$.
Notice that $\overline{W_g}$ is obtained from $W \times \IZ$
by identifying $(g(x),n)$ and $(x,n+1)$ for $x \in \partial_1W$ and $n \in \IZ$
and the map $l_1 \colon \overline{W_g} \to \overline{W_g}$ is induced by the map
$(x,n) \mapsto (x,n+1)$. 
The inclusion $l \colon \partial_0W \to W_g$ lifts uniquely to an inclusion
$\overline{l} \colon \partial_0W \to \overline{W_g}$ which satisfies
$\overline{e} \circ \overline{l}(W_g) = \{0\}$ and is a homotopy equivalence. 
In the model above this corresponds
to sending $x$ to $(x,0)$ for $x \in \partial_0W$.  
Now 
$$\Theta(f_g)  = l_*\bigl(\tau(g \circ i_1^{-1} \circ i_0)\bigr) $$
with respect to the simple structure on $\partial_0W$ coming from the structure
of a topological manifold follows Remark~\ref{rem:mapping_tori} 
and Lemma~\ref{lem:tauprime(f)_and_T_l_1} 
since $(g \circ i_1^{-1} \circ i_0) \simeq \overline{l}^{-1}  \circ l_1 \circ \overline{l}$.
Now the assertion~\ref{lem_theta_and_tau_fib_for_h_cobordism:Theta} follows from
\begin{eqnarray*}
l_*\bigl(\tau(g \circ i_1^{-1} \circ i_0)\bigr)
& = &
l_*(\tau(g)) + (l\circ g)_*(\tau(i_1^{-1})) +  (l\circ g \circ i_1^{-1})_*(\tau(i_0))
\\
& = &
0  - l_* \circ (i_0)_*^{-1} \circ (i_1)_* \circ 
(i_1)_*^{-1}(\tau(i_1)) +  l_* \circ (i_0)_*^{-1}(\tau(i_0))
\\
& = &
l_* \bigl(-(i_0)_*^{-1} \circ (i_1)_* \circ 
(i_1)_*^{-1}(\tau(i_1)) +  (i_0)_*^{-1} (\tau(i_0))\bigr)
\\
& = &
l_* \bigl(- (-1)^{\dim(\partial_0 W)} \cdot  \ast(\tau(W)) + \tau(W)\bigr)
\\
& = &
l_* \circ \bigl((-1)^{\dim(W)} \cdot  \ast + \id\bigr)(\tau(W))
\\
& = &
\bigl((-1)^{\dim(W)} \cdot  \ast + \id\bigr)\circ l_*(\tau(W)).
\end{eqnarray*}~%
\ref{lem_theta_and_tau_fib_for_h_cobordism:tauprime}
Consider the commutative diagram
$$\xymatrix@C=25mm{
\partial_0W  \ar[d]^-{\id} & \partial_0W \coprod \partial_0W
\ar[l]^-{\id \coprod g \circ i_1^{-1} \circ i_0}
\ar[r]^-{j}
\ar[d]^-{\id  \coprod i_1^ {-1} \circ i_0}
& \partial_0 W \times [0,1]\ar[d]^-{h \circ (i_0 \times \id_{[0,1]})}
\\
\partial_0W &
\partial_0W \coprod \partial_1W
\ar[l]^-{\id \coprod g}
\ar[r]^-{i_0 \coprod i_1}
& W
}$$
where $j\colon \partial_0W \coprod \partial_0W = 
\partial_0W \times \{0,1\} \to \partial_0 W \times [0,1]$
is the inclusion, $i_1^{-1} \colon W \to \partial_1 W$ is a homotopy inverse of $i_1$ and
$h \colon \id_W \simeq i_1 \circ i_1^{-1}$ is some homotopy.
The pushout of the upper row is the mapping 
torus $T_{g \circ i_1^{-1} \circ i_0}$. The pushout
of the lower row is $W_g$ and the structure $\xi^{\topo}(W_g)$ on the closed manifold $W_g$
is just the pushout of the simple structures. 
All vertical arrows are homotopy equivalences. We obtain a homotopy equivalence
$$\lambda \colon T_{g \circ i_1^{-1} \circ i_0} \to W_g.$$ 
We conclude from Remark~\ref{rem:mapping_tori} and Lemma~\ref{lem:tauprime(f)_and_T_l_1} 
since $\overline{l} \circ (g \circ i_1^{-1} \circ i_0) \simeq l_1 \circ \overline{l}$.
$$\tau'_{\fib}(f_g) = \tau(\lambda).$$
Let $\pr \colon W \to W_g$ be the canonical projection.
We conclude from Lemma~\ref{lem:properties_of_Whitehead_torsion},
assertion~\ref{lem_theta_and_tau_fib_for_h_cobordism:Theta} and the diagram above
\begin{eqnarray*}
\tau_{\fib}'(f_g) & = & \tau(\lambda)
\\
& = &
- (\pr \circ i_1)_*(\tau(i_1^{-1} \circ i_0)) + \pr_*(\tau(h \circ (i_0 \times \id_{[0,1]})))
\\
& = &
- (\pr \circ i_1)_*(\tau(i_1^{-1})) -
(\pr \circ i_1 \circ i_1^{-1})_*(\tau(i_0)) + \pr_*(\tau(h)) 
\\
& & \hspace{60mm} + (\pr \circ h)_*(\tau(i_0 \times \id_{[0,1]}))
\\
& = &
(\pr \circ i_1 \circ i_1^{-1})_*(\tau(i_1)) -  
\pr_*(\tau(i_0)) + \pr_*(\tau(\id_W)) +  \pr_*(\tau(i_0))
\\
& = &
(\pr \circ i_0)_* \circ (i_0^{-1} \circ i_1 \circ i_1^{-1})_*(\tau(i_1))
\\
 & = &
l_* \circ (i_0^{-1} \circ i_1)_*\bigl(\tau(W^*)\bigr)
\\
& = &
(-1)^{\dim(\partial_0 W)} \cdot  l_*\circ \ast(\tau(W))
\\
& = &
(-1)^{\dim(\partial_0 W)} \cdot  \ast \circ l_*(\tau(W))
\\
& = &
l_* (\tau(W)) - \Theta(f_g).
\end{eqnarray*}%
\ref{lem_theta_and_tau_fib_for_h_cobordism:tau}
This follows from Lemma~\ref{lem_thetaprime_compared_with_Theta_and_tau}~%
~\ref{lem_thetaprime_compared_with_Theta_and_tau:tau_from_tauprime} 
and assertion~\ref{lem_theta_and_tau_fib_for_h_cobordism:tauprime}.
\\[1mm]~%
\ref{lem_theta_and_tau_fib_for_h_cobordism:equality_of_vanishing} 
This follows from assertions~\ref{lem_theta_and_tau_fib_for_h_cobordism:Theta},%
\ref{lem_theta_and_tau_fib_for_h_cobordism:tauprime} 
and~\ref{lem_theta_and_tau_fib_for_h_cobordism:tau}.
\end{proof}

Hence $\Theta(f)$ and $\tau_{\fib}(f)$ are given in terms of
$\tau(W)$.  The map induced by $l$ on the fundamental groups can be
identified with the inclusion of $\pi_1(\partial_0 W)$ into the
semi-direct product $\pi_1(\partial_0 W) \rtimes_\phi \IZ =
\pi_1(W_g)$, where $\phi$ is the automorphism of $\pi_1(\partial_0 W)$
induced by $g \circ i_1^{-1} \circ i_0$. The map $l_* \colon
\Wh(\pi_1(W)) \to \Wh(\pi(W_g))$ is injective if $\phi = \id$ but not
in general. So it can happen that the $h$-cobordism $W$ is non-trivial
but both elements $\Theta(f_g)$ and $\tau_{\fib}(f_g)$
vanish. Moreover, for a fixed $h$-cobordism $W$ the answer to the
question whether $\Theta(f_g)$ or $\tau_{\fib}(f_g)$ vanishes, does
in general depend on $\pi_1(g)$.

\typeout{-----------------------  Section 8 ------------------------}

\section{Comparison with Farrell's obstruction over $S^1$}
\label{sec:comparison_with_Farrells_obstruction_over_S1}

In this section we show in the case of $S^1$ as base space that the
torsion obstructions defined in this article are equivalent to the
ones defined by Farrell~\cite{Farrell(1971)}. As in the paper~\cite{Farrell(1971)}
we will work in the smooth category in this section.

Throughout this section we consider a map
$$f \colon M \to S^1$$
from a connected closed smooth manifold $M$ to $S^1$ such that its homotopy
fiber has the homotopy type of a finite $CW$-complex, the
homomorphism $\pi_1(f) \colon \pi_1(M) \to \pi_1(S^1)$ is surjective
and the dimension of $M$ is at least five. 
Let $G$ be the kernel of $\pi_1(f)$. Choose an element $t \in
\pi_1(M)$ which is sent under $\pi_1(f)$ to the standard generator of
the infinite cyclic group $\pi_1(S^1)$.  Conjugation with $t$ induces
an automorphism $\alpha \colon G \xrightarrow{\cong} G$.  We obtain an
isomorphism $G \rtimes_{\alpha} \IZ \xrightarrow{\cong} \pi_1(M)$
which is the identity on $G$ and sends $1 \in \IZ$ to $t \in
\pi_1(M)$. We will use it to identify $G \rtimes_{\alpha} \IZ = \pi_1(M)$.

A \emph{splitting} of $M$ with respect to $f$ is a pair $(N,\nu)$ such
that $N$ is a closed submanifold of $M$ of codimension one together
with a framing $\nu$ of the normal bundle such that under the
Pontrjagin Thom construction $\nu$ corresponds to $f$.  Such a
splitting can be obtained by changing $f$ in its homotopy class to a
smooth map which is transversal to $\{e(0)\} \in S^1$ and taking $N$ to
be the preimage of $e(0)$.  If we take out a tubular neighborhood of
$N$ in $M$, we obtain a cobordism $M_N$ with two boundary components
$\partial_0M_N = N$ to $\partial_1M_N = N$. A splitting is
called a \emph{pseudo fibering} if $\bigl(M_N,\partial_0M_N, \partial_1M_N\bigr)$ 
is an $h$-cobordism. We use the convention that going from
$\partial_0M_N$ to $\partial_1M_N$ corresponds 
to going in the circle in the anticlockwise sense.

Farrell~\cite[Chapter~III]{Farrell(1971)} introduces an element 
$c(f)\in C(\IZ G; \alpha)$ which depends only on the homotopy class of $f$.
It vanishes if and only if there exists a pseudo fibering
(see~\cite[Chapter~V]{Farrell(1971)}). Farrell~\cite[Chapter~IV]{Farrell(1971)}
constructs a duality isomorphism
$C(\IZ G; \alpha) \xrightarrow{\cong} C(\IZ G; \alpha^{-1})$ which sends
$c(f)$ to $c(\con \circ f)$. Hence the vanishing of $c(f)$ is equivalent 
to the vanishing of $c(\con \circ f)$.

Farrell~\cite[Chapter~I]{Farrell(1971)} defines a map $p \colon \Wh(G
\times_{\alpha} \IZ) \to C(\IZ G,\alpha)$.  By inspecting a
highly connected splitting one sees that it sends $\tau'_{\fib}(f)$ to
$c( f)$ (see~\cite[Lemma~3.7]{Farrell(1971)}).
In particular the vanishing of $\tau'_{\fib}(f)$
implies the vanishing of $c(f)$.

Now suppose that $c(f)$ is trivial. Then we can find a pseudo fibering
$(N,\nu)$ for $f$. Recall that associated to it is an $h$-cobordism
$M_N$ obtained from $M$ by deleting a tubular neighborhood of $N$.
Its Whitehead torsion $\tau(M_N)$ lives in $\Wh(\pi(\partial_0 M_N))$.  Let
$i \colon \partial_0 M_N \to M$ be the inclusion and $\pr \colon
\Wh(G) \to \Wh(G) \otimes_{\alpha} \IZ$ be the canonical projection.
Then Farrell~\cite[Chapter~VI]{Farrell(1971)} defines
$$\tau(f) \in \Wh_{\alpha}(G):=\Wh(G)\otimes_\alpha \IZ$$
to be the image of $\tau(M_N)$ under the map $\pr \circ i_*
\colon \Wh(\partial_0 M_N) \to \Wh(G) \otimes_{\alpha} \IZ$.  The
inclusion $G\to G \rtimes_{\alpha} \IZ$ induces a map
$$j \colon \Wh(G) \otimes_{\alpha} \IZ \to \Wh(G \rtimes_{\alpha} \IZ)$$
which is injective by~\cite{Farrell-Hsiang(1970)}.

The identity on $N$ yields a diffeomorphism $g \colon \partial_0M_N
\xrightarrow{\cong} \partial_1M_N$ and we can consider in the notation
of Section~\ref{sec:Gluing_h-cobordisms} the manifold $(M_N)_g$
together with a up to homotopy well-defined map $f' \colon M_N \to
S^1$. We can construct a diffeomorphism $\psi \colon (M_N)_g \to M$
such that up to homotopy $f \circ \psi = f'$. Now we conclude from
Lemma~\ref{lem_theta_and_tau_fib_for_h_cobordism}~%
~\ref{lem_theta_and_tau_fib_for_h_cobordism:tauprime} that the injective map 
$j \colon \Wh(G) \otimes_{\alpha} \IZ \to \Wh(G \rtimes_{\alpha} \IZ)$ 
sends $\tau(f)$ to 
$(-1)^{\dim(W)} \cdot \ast\bigl(\tau_{\fib}'(f)\bigr)$. 
Hence $\tau(f)$ vanishes if and only if $\tau'_{\fib}(f)$ vanishes.

Thus we have shown that the vanishing of $\tau_{\fib}'(f)$ implies
the vanishing of the obstructions $c(f)$ and $\tau(f)$ of Farrell.
Exploiting the main theorem of Farrell~\cite{Farrell(1971)}
that $c(f)$ and $\tau(f)$ vanish if and only if $f$ is 
homotopic to a smooth fiber bundle, we conclude
from Theorem~\ref{the:prop._of_fiber_torsion_obstruction}~%
\ref{the:prop._of_fiber_torsion_obstruction:obstruction}, 
Lemma~\ref{lem_theta_and_tau_fib_for_h_cobordism}~%
\ref{lem_theta_and_tau_fib_for_h_cobordism:equality_of_vanishing}:

\begin{theorem}
\label{the:Comparison_with_Farrell}
Let $f \colon M \to S^1$ be a map from a connected closed smooth manifold $M$
to $S^1$ such that its homotopy fiber has the homotopy type of a
finite $CW$-complex, the homomorphism $\pi_1(f) \colon \pi_1(M) \to \pi_1(S^1)$ 
is surjective and the dimension of $M$ is at least five.
Then the  following assertions are equivalent:
\begin{enumerate}
\item $\tau_{\fib}'(f)$  vanishes;
\item $\Theta(f)$ and $\tau_{\fib}(f)$ vanish;
\item $c(f)$ and $\tau(f)$ vanish;
\item $f$ is homotopic to a smooth fiber bundle.
\end{enumerate}
\end{theorem}

\begin{remark} \label{rem_Siebemann_topological_category}
Siebenmann~\cite[Section~1]{Siebenmann(1970_total)} 
says that the main theorem of Farrell~\cite{Farrell(1971)} and hence
Theorem~\ref{the:Comparison_with_Farrell} hold also in the topological category.
\end{remark}

\typeout{-----------------------  Section 9 ------------------------}

\section{A composition formula}
\label{sec:a_composition_formula}

In this section we want to to express $\tau_{\fib}(g \circ f)$ in terms of
  $\tau_{\fib}(f)$ and $\tau_{\fib}(g)$.

  Let $f \colon M \to N$ and $g \colon N \to B$
  be maps of closed path-connected manifolds. Assume that the homotopy
  fibers of both $f$ and $g$ have the homotopy type of a finite
  $CW$-complex.  Then the same is true for the composite $g \circ f$
  since there is a fibration $\hofib(f) \to \hofib(g \circ f) \to
  \hofib(g)$. So the elements $\Theta(f) \in H^1\bigl(N,\Wh(\pi(M))\bigr)$,
  $\Theta(g) \in H^1\bigl(N,\Wh(\pi(N))\bigr)$ and $\Theta(g \circ f) \in
  H^1\bigl(B;\Wh(\pi(M))\bigr)$ are defined. 
  Assume that $\Theta(f)$, $\Theta(g)$ and $\Theta(g \circ f)$
  vanish.  We obtain fiber torsion obstructions 
  (see~Definition~\ref{def:fiber_torsion_obstruction}) 
  \begin{eqnarray*}
  \tau_{\fib}(f) 
  & \in & 
  \cok\bigl(\chi(N) \cdot (\mu_f \circ i_f)_*\colon 
  \Wh(\pi(\hofib(f))) \to \Wh(\pi(M))\bigr);
  \\
  \tau_{\fib}(g) 
  & \in & 
  \cok\bigl(\chi(B) \cdot (\mu_g \circ i_g)_*\colon 
  \Wh(\pi(\hofib(g))) \to \Wh(\pi(N))\bigr);
  \\
  \tau_{\fib}(g \circ f) 
  & \in & 
  \cok\bigl(\chi(B) \cdot (\mu_{g \circ f} \circ i_{g \circ f})_*\colon 
  \Wh(\pi(\hofib(g \circ f))) \to \Wh(\pi(M))\bigr),
  \end{eqnarray*}
  where $i_f$, $i_g$ and $i_{g \circ f}$ denote the inclusions of the homotopy fibers
  over $z \in N$ or $f(z) \in B$ respectively for a fixed base point $z \in N$.

  Given a simple fibration $p \colon E \to B$ over a finite $CW$-complex 
  one can define a \emph{transfer map} 
  \begin{eqnarray}
  p^* \colon \Wh(B) & \to & \Wh(E)
  \label{transfer_map_past}
  \end{eqnarray}
  as follows, provided that the fiber is homotopy
  equivalent to a finite $CW$-complex. For simplicity assume that
  $B$ is path connected, the general case is then done componentwise.
  Given an element $\omega \in \Wh(\pi(B))$, choose a finite $CW$-complex
  $X$ together with a homotopy equivalence $f \colon  X \to B$ such 
  $\tau(f) = \omega$. Consider the pullback diagram
  $$\xymatrix{f^*E \ar[r]_{\overline{f}} \ar[d]_{f^*p} & E \ar[d]_{p}
  \\X \ar[r]_{f} & B
  }
  $$
  Choose a point $x \in X$. Let $f(x) \in B$ be its image under $f$. 
  Choose a simple structure  on the fiber of $E$ over $f(x)$. 
  Use the same simple structure $\zeta$ on the fiber of $p^*E$ over $x$. 
  We obtain well-defined simple structures $\xi(x,\zeta)$ on $f^*E$ and    
  $\xi(f(x),\zeta)$ on $E$ (see Notation~\ref{not:_xi(E,b,xi(F_b))}). Now define
  $$p^*(\omega) = \tau\bigl(\overline{f}\colon 
  (p^*E,(\xi(x,\zeta)) \to (E,(\xi(f(x),\zeta))\bigr).$$
  This is a well-defined homomorphism because of 
  Lemma~\ref{lem:properties_of_Whitehead_torsion},
  Lemma~\ref{lem:compatibility_of_simple_structure_with_pushouts} and 
  Lemma~\ref{lem:compatibility_of_simple_structure_with_simple_pullback}.
  By construction the transfer is compatible with pullbacks and by
  Lemma~\ref{lem:simple_structure_and_fiber_homotopy_equivalences} 
  with fiber homotopy equivalences.
  More information about this transfer map including its algebraic description 
  and computational tools can be found for instance 
  in~\cite{Lueck(1984)}, \cite{Lueck(1986)}  and~\cite{Lueck(1987)}.

  We obtain a transfer map 
  $$f^* \colon \Wh(\pi(N)) \to \Wh(\pi(M))$$
  from the transfer associated in~\eqref{transfer_map_past} 
  to the fibration $\widehat{f} \colon \FIB(f) \to N$ and the
  isomorphism $(\mu_f)_* \colon 
  \Wh\bigl(\pi(\FIB(f))\bigr) \xrightarrow{\cong} \Wh(\pi(M))$ induced by the homotopy
  equivalence $\mu_f \colon \FIB(f) \to M$. Since the transfer is compatible 
  with pullbacks and fiber homotopy equivalences, the transfer 
  induces a map, also denoted by $f^*$,
  \begin{multline}
  f^* \colon \cok\bigl(\chi(N) \cdot (\mu_g \circ i_g)_*\colon 
  \Wh(\pi(\hofib(g))) \to \Wh(\pi(N))\bigr)
  \\ \to 
   \cok\bigl(\chi(B) \cdot (\mu_{g \circ f} \circ i_{g \circ f})_*\colon 
  \Wh(\pi(\hofib(g \circ f))) \to \Wh(\pi(M))\bigr).
  \label{transfer_map_fast}
  \end{multline}
  
Since $\chi(N)=\chi(M)\cdot\chi(B)$, the element $\tau_{\fib}(f)$ defines an element 
 $\overline{\tau}_{\fib}(f)\in \cok\bigl(\chi(B) 
 \cdot (\mu_{g \circ f} \circ i_{g \circ f})_*\colon 
  \Wh(\pi(\hofib(g \circ f))) \to \Wh(\pi(M))\bigr)$.

\begin{theorem}\label{the:composition_formula_for_tau} 
Under the conditions above we get 
  $$\tau_{\fib}(g \circ f) = \overline{\tau}_{\fib}(f) + f^*(\tau_{\fib}(g))$$
in $\cok\bigl(\chi(B) \cdot (\mu_{g \circ f} \circ i_{g \circ f})_*\colon 
  \Wh(\pi(\hofib(g \circ f))) \to \Wh(\pi(M))\bigr)$.

If we additionally assume $\chi(B) = 0$, we get 
$$\tau_{\fib}(g \circ f) = \tau_{\fib}(f) + f^*(\tau_{\fib}(g))$$
in $\Wh(\pi(M))$.
\end{theorem}
\begin{proof}
Consider the following commutative diagram
$$\xymatrix{
\mu_g^*\FIB(f) \ar[r]^-{\overline{\mu_g}} \ar[rd]_-{\mu_g^*\widehat{f}}
\ar@/^{2em}/[rrr]^-{\phi}
& \FIB(f) \ar[r]^-{\mu_f} \ar[dr]^-{\widehat{f}} 
&  M \ar[d]^-{f} &
\FIB(g \circ f) \ar[ldd]^-{\widehat{g \circ f}} \ar[l]_-{\mu_{g \circ f}}
\\
& \FIB(g) \ar[r]^-{\mu_g} \ar[dr]_-{\widehat{g}}
& N \ar[d]^-{g}
\\
& & B & 
}
$$
Here $\mu_g^*\widehat{f} \colon \mu_g^*\FIB(f) \to \FIB(g)$ is the pullback of the
fibration $\widehat{f} \colon \FIB(f) \to N$ with $\mu_g \colon \FIB(g) \to N$
and $\phi$ is an appropriate fiber homotopy equivalence 
of fibrations over $B$ from 
$\widehat{g} \circ \mu_g^*\widehat{f}$ to $\widehat{g \circ f}$
such that $\mu_{g \circ f} \circ \phi$
is homotopic to $\widehat{\mu_f} \circ \overline{\mu_g}$. 
Since $\widehat{g \circ f}$ is simple, the fibration 
$\widehat{g} \circ \mu_g^*\widehat{f}$
is also simple

We equip $M$, $N$ and $B$ with canonical simple structure coming 
from the manifold structure. 

Choose simple structures on $\hofib(f)$, $\hofib(g)$, $\hofib(g \circ f)$
and the fiber of the fibration given by the composite $\widehat{g} \circ \mu_g^*\widehat{f}$.
We equip the total spaces of the simple fibrations over finite $CW$-complexes
$\widehat{g} \colon \FIB(g)\to B$, $\widehat{f} \colon \FIB(f) \to N$,
$\widehat{g \circ f} \colon \FIB(g \circ f)\to B$
and $\widehat{g} \circ \mu_g^*\widehat{f} \colon \mu_g^*\FIB(f) \to B$
with the simple structure coming from Notation~\ref{not:_xi(E,b,xi(F_b))}.

Since $\phi$ is a fiber homotopy equivalence, we conclude from 
Lemma~\ref{lem:compatibility_of_simple_structure_with_pushouts}
$$\tau(\phi) = 0$$
in $\cok\bigl(\chi(B) \cdot (\mu_{g \circ f} \circ i_{g \circ f})_*\colon 
\Wh(\pi(\hofib(g \circ f)) \to \Wh(\pi(M))\bigr)$. 

There is a second simple structure on  $\mu_g^*\FIB(f)$
which comes from Remark~\ref{rem:simple_structures_on_B_and_F-yield_one_on_E} 
applied to the fibration $\mu_g^*\widehat{f} \colon \mu_g^*\FIB(f) \to \FIB(g)$
and the simple structure defined on $\FIB(g)$ above. 
With respect to this simple structure we conclude from the definition  of the transfer maps
$$f^*(\tau(\mu_g)) = 
(\mu_f)_*\bigl(\tau(\overline{\mu_g})\bigr).$$
These two simple structures on  $\mu_g^*\FIB(f)$
are not necessarily the same. But a modification of the proof
of Lemma~\ref{lem:simple_structure_and_fiber_homotopy_equivalences} and
Lemma~\ref{lem:compatibility_of_simple_structure_with_pushouts} show
that the image of the Whitehead torsion of the identity map on $\mu_g^*\FIB(f)$
under the isomorphism 
$$(\mu_f \circ \overline{\mu_g})_* \colon 
\Wh\bigl(\pi(\mu_g^*\FIB(f))\bigr) \xrightarrow{\cong} \Wh(M)$$
with respect to these two different simple structures becomes zero when regarded in 
$\cok\bigl(\chi(B) \cdot (\mu_{g \circ f} \circ i_{g \circ f})_*\colon 
\Wh(\pi(\hofib(g \circ f)) \to \Wh(\pi(M))\bigr)$.
Hence it does not matter which simple structure we use.

From the composition formula for Whitehead torsion and the equalities above 
we conclude in $\cok\bigl(\chi(B) \cdot (\mu_{g \circ f} \circ i_{g \circ f})_*\colon 
\Wh(\pi(\hofib(g \circ f))) \to \Wh(\pi(M))\bigr)$
\begin{eqnarray*}
\tau_{\fib}(g \circ f) 
& = & 
\tau(\mu_{g \circ f})
\\
& = & 
\tau(\mu_{g \circ f}) + (\mu_{g \circ f})_*(\tau(\phi))
\\
& = & 
\tau\bigl(\mu_{g \circ f} \circ \phi\bigr)
\\
& = & 
\tau\bigl(\mu_f \circ \overline{\mu_g}\bigr)
\\
& = & 
\tau(\mu_f) + (\mu_f)_*\bigl(\tau(\overline{\mu_g})\bigr)
\\
& = & 
\tau(\mu_f) + f^*\bigl(\tau(\mu_g)\bigr)
\\
& = & 
\bar\tau_{\fib}(f) + f^*(\tau_{\fib}(g)).\qedhere
\end{eqnarray*}
\end{proof}

\typeout{-----------------------  Section 10 ------------------------}

\section{Poincar\'e torsion}
\label{sec:Poincare_torsion}

The definitions and probably most of the properties of the Poincar\'e torsion
are known to the experts but since we could not find a good reference in the literature,
we elaborate on them in this section. Some information can be found for 
instance in~\cite[Proposition~26]{Korzieniewski(2005)}.

Let $X$ be a finite $CW$-complex.
Suppose that $X$ is connected. Denote by
$\pi$ the fundamental group $\pi_1(X)$.
Let $p\colon \widetilde{X} \to X$ be the universal covering.
Denote by $C_*(\widetilde{X})$ the cellular $\IZ \pi$-chain complex. 
Let $C^{n-*}(\widetilde{X})$
denote the dual $\IZ\pi$-chain complex, where we always use the
involution on $\IZ \pi$ sending $g$ to $w(g) \cdot g^{-1}$
for a given homomorphism $w \colon \pi_1(X) \to \{\pm 1\}$.
We call $X$ an \emph{$n$-dimensional Poincar\'e complex},
if there exist a so called 
\emph{orientation homomorphism} $w = w_1(X) \colon \pi_1(X) \to \{\pm 1\}$
and an element called \emph{fundamental class}
$[X] \in H_n(X;\IZ^w)$ and 
such that the up to $\IZ \pi$-chain homotopy uniquely defined $\IZ \pi$-chain map
\begin{eqnarray}
& - \cap [X]\colon C^{n-*}(\widetilde{X})
\to C_*(\widetilde{X}) &
\label{Poincare_chain_duality_map:closed_case}
\end{eqnarray}
is a $\IZ \pi$-chain homotopy equivalence. Here and in the sequel $\IZ^w$ is the $\IZ\pi$-module
whose underlying abelian group is $\IZ$ and for which $g \in \pi$ acts by multiplication with $w(g)$.
If a finite $CW$-complex
carries some structure of a Poincar\'e complex, then
$H_n(X;\IZ^w)$ is infinite cyclic and the fundamental class $[X]$ is a generator and hence
unique up to sign, and one can rediscover
the orientation homomorphism $w$ from $X$ as a $CW$-complex
(see~\cite[paragraph before~1.3]{Lueck-Ranicki(1992)}).

If $X$ is not connected, we require that each component
$C \in \pi_0(X)$ is an
$n$-dimensional Poincar\'e complex in the above sense.

Let $(X,\partial X)$ be a finite $CW$-pair such that $X$ is
$n$-dimensional and $\partial X$ is $(n-1)$-dimensional.
Suppose that $X$ is connected. Denote by
$\pi$ the fundamental group $\pi_1(X)$.
Let $p\colon \widetilde{X} \to X$ be the universal covering and put
$\widetilde{\partial X} = p^{-1}(\partial X)$.
Denote by $C_*(\widetilde{X})$ and
$C_*(\widetilde{X},\widetilde{\partial X})$ the
cellular $\IZ \pi$-chain complexes. Let $C^{n-*}(\widetilde{X})$
and $C^{n-*}(\widetilde{X},\widetilde{\partial X})$
denote the dual $\IZ\pi$-chain complexes, where we always use the
involution on $\IZ \pi$ sending $g$ to $w(g) \cdot g^{-1}$
for a given homomorphism $w \colon \pi_1(X) \to \{\pm 1\}$.
We call $(X,\partial X)$ a \emph{$n$-dimensional Poincar\'e pair},
if there exists  a  so called \emph{orientation homomorphism} $w = w_1(X) \colon \pi_1(X) \to \{\pm 1\}$
and an element called \emph{fundamental class}
$[X,\partial X] \in H_n(X,\partial X;\IZ^w)$ 
such that the up to $\IZ \pi$-chain homotopy uniquely defined $\IZ \pi$-chain map
\begin{eqnarray}
& - \cap [X,\partial X]\colon C^{n-*}(\widetilde{X})
\to C_*(\widetilde{X},\widetilde{\partial X})
\label{Poincare_chain_duality_map}
\end{eqnarray}
is a $\IZ \pi$-chain homotopy equivalence
and $\partial X$ is a Poincar\'e complex with respect to the fundamental classes
of each component of $\partial X$ coming from the image of the 
fundamental class of $X$ under the boundary
homomorphism $H_n(X,\partial X;\IZ^w) \to H_{n-1}(\partial X;\IZ^w)$.
 
If $X$ is not connected, we require that for each component
$C \in \pi_0(X)$ the pair $(C,C \cap \partial X)$ is an
$n$-dimensional Poincar\'e pair in the sense above. (To simplify notation, we use the symbol $w$ for the orientation homomorphisms on all the Poincar\'e complexes and pairs occuring.)

The chain complexes
$C^{n-*}(\widetilde{X})$ and $C_*(\widetilde{X},\widetilde{\partial X})$
inherit from the $CW$-structure a cellular $\IZ \pi$-basis which is unique
up to permuting the elements of the basis or multiplying with
elements of the form $\pm g$ for $g \in \pi$. Hence one can associate
to the $\IZ\pi$-chain homotopy equivalence defined
in~\eqref{Poincare_chain_duality_map}  
its Whitehead torsion $\tau\bigl(\cap[X,\partial X]\bigr) \in \Wh(\pi)$. 
Since $X$ is connected and hence
$H^0(X) = \IZ$, we get from Poincar\'e duality
that $H_n(X,\partial X;\IZ^w) = \IZ$ and $[X,\partial X]$ must be a generator.
If we replace $[X,\partial X]$ by $-[X,\partial X]$, we get
$\tau\bigl(- \cap (-[X,\partial X])\bigr) = \tau\bigl(- \cap [X,\partial X]\bigr)$ since
the Whitehead torsion satisfies the  composition formula
$\tau(g \circ f) = \tau(f) + \tau(g)$ and $\tau(-\id) = 0$.

\begin{definition} \label{def:Poincare_torsion}
Let $(X,\partial X)$ be an $n$-dimensional Poincar\'e pair. If $X$
is connected, define its Poincar\'e torsion
$$\rho(X,\partial X) \in \Wh(\pi(X))$$
by the Whitehead torsion 
$\tau\bigl(- \cap [X,\partial X] \colon C^{n-*}(\widetilde{X})
\to C_*(\widetilde{X},\widetilde{\partial X})\bigr)$ for any choice
of fundamental class $[X,\partial X] \in H_n(X,\partial X;\IZ^w)$.

If $X$ is not connected, define
$$\rho(X,\partial X) \in \Wh(\pi(X))= \bigoplus_{C \in \pi_0(X)} \Wh(\pi(C))$$
by the various elements $\rho(C,C \cap \partial X) \in \Wh(\pi(C))$.

We call an $n$-dimensional Poincar\'e pair $(X,\partial X)$ \emph{simple}
if $\rho(X,\partial X) = 0$.
\end{definition}

Next we collect the basic properties of this invariant
(see also~\cite[Proposition~2.7]{Wall(1999)}).
Notice that because of Theorem~\ref{the:prop_of_Poinc._tor}~%
\ref{the:prop_of_Poinc._tor:homotopy_invariance}
we can extend the Definition~\ref{def:Poincare_torsion} of $\rho(X,A)$ 
to pairs of spaces $(X,A)$ with simple structures 
for which there exists a simple homotopy equivalence
$(X,A) \to (Y,\partial Y)$ with a finite Poincar\'e pair as target.
This applies in particular to $(X,A) = (M,\partial M)$ for a compact topological
manifold $M$ with boundary $\partial M$.

\begin{theorem} \label{the:prop_of_Poinc._tor}
\begin{enumerate}

\item \label{the:prop_of_Poinc._tor:behaviour_under_involution}
If $X$ is an $n$-dimensional Poincar\'e complex, then
$$\rho(X) = (-1)^n \cdot \ast\bigl(\rho(X)\bigr),$$
where $\ast \colon \Wh(\pi(X)) \to \Wh(\pi(X))$ is the 
$w_1(X)$-twisted involution. 

More generally, we get for an $n$-dimensional Poincar\'e pair $(X,\partial X)$
$$(j_{\partial X})_*\bigl(\rho(\partial X)\bigr)
= (-1)^n \cdot \ast\bigl(\rho(X,\partial X)\bigr) - \rho(X,\partial X),$$
where $j_{\partial X} \colon \partial X \to X$ is the inclusion;

\item \label{the:prop_of_Poinc._tor:homotopy_invariance}
If $(f,\partial f)\colon (X,\partial X) \to (Y,\partial Y)$
is a homotopy equivalence of $n$-dimensional Poincar\'e pairs, then
$$\hspace{13mm}
\rho(Y,\partial Y) - f_*\bigl(\rho(X,\partial X)\bigr)
\; = \;
\tau(f) + (-1)^n \cdot \ast\bigl(\tau(f)\bigr) - 
(j_{\partial M})_*\bigl(\tau(\partial f)\bigr),
$$
where $f_*\colon \Wh(\pi(X)) \to \Wh(\pi(Y))$ 
and $(j_{\partial M})_*\colon \Wh(\pi(\partial Y)) \to \Wh(\pi(Y))$
are the homomorphisms induced by $f\colon X \to Y$ and the inclusion
$j_{\partial M}\colon \partial Y \to Y$, $\tau(f) \in \Wh(\pi(Y))$ and 
$\tau(\partial f) \in \Wh(\pi(\partial Y))$ denote the Whitehead torsion
of the homotopy equivalences of finite $CW$-complexes
$f$ and $\partial f$, and $\ast \colon \Wh(\pi(Y)) \to \Wh(\pi(Y))$ is the 
$w_1(Y)$-twisted involution;

\item \label{the:prop_of_Poinc._tor:glueing_formula}
Let $(X,\partial X)$ and $(Y,\partial Y)$ be $n$-dimensional Poincar\'e
pairs such that $X$ and $Y$ are connected. Let
$f\colon \partial X \to \partial Y$ be a homotopy equivalence.
Let $X \cup_f Y$ be the space obtained by gluing $X$ to $Y$ along $f$.
Denote by $j_X\colon X \to X\cup_f Y$, $j_Y\colon Y \to X\cup_f Y$  
and $j_{\partial Y}\colon \partial Y \to X\cup_f Y$ the canonical inclusions. 
Then $X \cup_f Y$ is a connected $n$-dimensional Poincar\'e complex and
\begin{multline*}
\hspace{14mm} \rho(X\cup_f Y) \; = \;
(-1)^n \cdot \ast \circ (j_X)_*\bigl(\rho(X,\partial X)\bigr) 
+ (j_Y)_*\bigl(\rho(Y,\partial Y)\bigr)
\\
+ (j_{\partial Y})_*\bigl(\tau(f)\bigr),
\end{multline*}
where $\ast \colon \Wh\bigl(\pi(Z \cup_f Y)\bigr) \to \Wh\bigl(\pi(Z \cup_f Y)\bigr)$ 
is the $w_1(Z \cup_f Y)$-twisted involution.

\item \label{the:prop_of_Poinc._tor:product_formula}
Let $(X,\partial X)$ resp.  $(Y,\partial Y)$ be a $m$- resp. $n$-dimensional
Poincar\'e pair such that $X$ and $Y$ are connected. Then
$$(X,\partial X) \times (Y,\partial Y) =
(X \times Y,X \times \partial Y \cup \partial X \times Y)$$
is an $(n+m)$-dimensional Poincar\'e pair with
\begin{multline*}
\hspace{13mm} 
\rho\bigl((X,\partial X) \times (Y,\partial Y)\bigr) 
\; = \;\chi(X,\partial X) \cdot
(k_Y)_*\bigl(\rho(Y,\partial Y)\bigr)
\\ + \chi(Y,\partial Y) \cdot (k_X)_*\bigl(\rho(X,\partial X)\bigr),
\end{multline*}
where $k_X\colon X \to X \times Y$ and $k_Y\colon Y \to X \times Y$ are the
inclusions and $\chi$ denotes the Euler characteristic;

\item \label{the:prop_of_Poinc._tor:triviality_for_manifolds}
Let $M$ be a compact	 topological manifold
(possibly with boundary $\partial M$). Then
$$\rho(M,\partial M) = 0.$$

\end{enumerate}

\end{theorem}
\begin{proof}%
\ref{the:prop_of_Poinc._tor:behaviour_under_involution}
The $\IZ\pi$-chain map $-\cap[X]\colon C^{n-*}(\widetilde{X}) \to
C_*(\widetilde{X})$ is self-dual in the sense
that it is $\IZ\pi$-chain homotopic to the one obtained
from $-\cap[X]$ by applying the functor $C^{n-*}$ and the obvious
identification $(C^{n-*}(\widetilde{X}))^{n-*} = C_*(\widetilde{X})$.
This follows from the fact the chain map $-\cap[X]$ is the zeroth part
of a cocycle in the $Q^n$-group associated to $C_*(\widetilde{X})$
(see~\cite[Section~1 and Proposition~2.1]{Ranicki(1980b)}).   This implies
$$\rho(X)= \tau(-\cap [X]) = \tau\bigl(C^{n-*}(-\cap [X])\bigr) =
(-1)^n \cdot \ast\bigl(\tau(- \cap [X])\bigr) = (-1)^n \cdot \ast\bigl(\rho(X)\bigr).$$

The case of a pair is  more complicated.
For the definition of the mapping cylinder and cone of a chain map and the basic
properties of these we refer for instance to~\cite[page~213~ff.]{Lueck(1989)}.
We have the short based exact  sequences of finite based free $\IZ \pi$-chain complexes
$$0 \to C_*(\widetilde{\partial X}) \xrightarrow{i_*} C_*(\widetilde{X}) 
\xrightarrow{p_*} C_*(\widetilde{X},\widetilde{\partial X}) \to 0.$$
and
$$0 \to C^{n-*}(\widetilde{X},\widetilde{\partial X}) \xrightarrow{p^{n-*}} 
C^{n-*}(\widetilde{X}) 
\xrightarrow{i^{n-*}} C^{n-*}(\widetilde{\partial X}) \to 0.$$
We obtain short based exact sequences of  finite based free $\IZ \pi$-chain complexes
$$0 \to C_*(\widetilde{X})  \xrightarrow{j_*} \cyl(p_*) \xrightarrow{q_*} \cone(p_*)
\to 0$$
and
$$0 \to  \Sigma C_*(\widetilde{\partial X}) \xrightarrow{k_*} \cone(p_*) 
\xrightarrow{r_*} \cone\bigl(C_*(\widetilde{X},\widetilde{\partial X})\bigr) \to 0.$$
Since the chain map 
$0_* \to \cone\bigl(C_*(\widetilde{X},\widetilde{\partial X})\bigr)$
is a simple $\IZ \pi$-chain homotopy equivalence, the chain map 
$$k_* \colon \Sigma C_*(\widetilde{\partial X}) \xrightarrow{\simeq_s} \cone(p_*)$$
is a simple chain $\IZ \pi$-chain homotopy equivalence.

Analogously we get short based exact sequences of  
finite based free $\IZ \pi$-chain complexes
$$0 \to C^{n-*}(\widetilde{X},\widetilde{\partial X}) \xrightarrow{l_*} \cyl(p^{n-*}) 
\xrightarrow{r_*} \cone(p^{n-*}) \to 0$$
and
$$0 \to \cone\bigl(C^{n-*}(\widetilde{X},\widetilde{\partial X})\bigr)
\xrightarrow{m_*} \cone(p^{n-*})  \xrightarrow{s_*} 
C^{n-*}(\widetilde{\partial X}) \to 0$$
and 
$$s_* \colon \cone(p^{n-*})  \xrightarrow{\simeq_s} 
C^{n-*}(\widetilde{\partial X})$$
is a simple chain $\IZ \pi$-chain homotopy equivalence.

We obtain an up to chain homotopy commutative diagram 
of finite based free $\IZ\pi$-chain complexes
\xycomsquare{C^{n-*}(\widetilde{X},\widetilde{\partial X})}
{p^{n-*}} 
{C^{n-*}(\widetilde{X})}
{- \cap[X,\partial X]}
{- \cap[X,\partial X]}
{C_*(\widetilde{X})} 
{p_*} 
{C_*(\widetilde{X},\widetilde{\partial X})}
Actually there is a preferred chain homotopy which is unique up to higher homotopies
coming from the cocycle in the $Q^n$-group associated to $C_*(\widetilde{X})$
(see~\cite[Section~1 and Proposition~2.1]{Ranicki(1980b)}). 
Hence we obtain up to chain homotopy unique chain maps $\alpha_*$ and
$\beta_*$ making the following diagram with based exact rows commutative
$$
\xymatrix{
0 \ar[r] 
& C^{n-*}(\widetilde{X},\widetilde{\partial X}) \ar[r]_{l_*} \ar[d]_{- \cap[X,\partial X]} 
& \cyl(p^{n-*}) \ar[r]_{r_*} \ar[d]_{\alpha_*}
& \cone(p^{n-*}) \ar[d]_{\beta_*} \ar[r]
&  0
\\
0 \ar[r]
&
C_*(\widetilde{X})  \ar[r]_{j_*} 
&
\cyl(p_*) \ar[r]_{q_*} 
& \cone(p_*) \ar[r]
&  0
}
$$
Additivity of the Whitehead torsion implies
\begin{eqnarray*}
\rho(X,\partial X) 
& = & 
\tau\bigl(- \cap[X,\partial X] \colon C^{n-*}(\widetilde{X})
\to C_*(\widetilde{X},\widetilde{\partial X}) \bigr)
\\
& = & 
(-1)^n \cdot \ast\left(
\tau\bigl(- \cap[X,\partial X] \colon C^{n-*}(\widetilde{X},\widetilde{\partial X})
\to C_*(\widetilde{X})  \bigr)\right)
\\
& = & 
(-1)^n \cdot \ast\bigl(\tau(\alpha_*) - \tau(\beta_*)\bigr).
\end{eqnarray*}
Moreover, we obtain a commutative diagram
$$\xymatrix{
C^{n-*}(\widetilde{X}) \ar[r]_{u_*} \ar[d]_{- \cap[X,\partial X]}
&
\cyl(p^{n-*}) \ar[d]_{\alpha_*}
\\
C_*(\widetilde{X},\widetilde{\partial X}) \ar[r] \ar[r]_{v_*}
& \cyl(p_*)
}
$$
where $u_*$ and $v_*$ the canonical inclusions and simple chain homotopy equivalences.
From the composition formula for Whitehead torsion we conclude
\begin{eqnarray*}
\rho(X,\partial X) 
& = & 
\tau(\alpha_*).
\end{eqnarray*}
Finally we obtain an up to chain homotopy commutative diagram with
simple homotopy equivalences as rows
$$\xymatrix{
\cone(p^{n-*}) \ar[r]_{s_*} \ar[d]_{\beta_*}
&
\Sigma C^{n-1-*}(\widetilde{\partial X})  \ar[d]_{\Sigma(- \cap[\partial X])}
\\
\cone(p_*)
&
\Sigma C_*(\widetilde{\partial X}) \ar[l]^-{k_*}
}
$$
From the composition formula for Whitehead torsion we conclude
$$
(j_{\partial X})_*\bigl(\rho(\partial X)\bigr)
 =  
\tau\bigl(- \cap[\partial X]\bigr)
 = - \tau\bigl(\Sigma(- \cap[\partial X])\bigr)
 = -\tau(\beta_*).
$$
This implies
\begin{eqnarray*}
(j_{\partial X})_*\bigl(\rho(\partial X)\bigr)
& = & 
- \tau(\beta_*)
\\ 
& = & 
(-1)^n \cdot \ast\bigl(\rho(X,\partial X)\bigr) - \tau(\alpha_*)
\\
& = & 
(-1)^n \cdot \ast\bigl(\rho(X,\partial X)\bigr) - \rho(X,\partial X).
\end{eqnarray*}
This finishes the proof of 
assertion~\ref{the:prop_of_Poinc._tor:behaviour_under_involution}.
\\[2mm]%
\ref{the:prop_of_Poinc._tor:homotopy_invariance}
Obviously it suffices to treat the case, where $X$ and $Y$ are connected,
the general case follows componentwise.
Choose the fundamental classes such that
$H_n(f,\partial f)$ maps $[X,\partial X]$ to $[Y,\partial Y]$.
Then the following
diagram of $\IZ\pi$-chain complexes commutes where we identify
$\pi = \pi_1(X) = \pi_1(Y)$ by $\pi_1(f)$
$$\xymatrix@!C=8em{
C^{n-*}(\widetilde{X}) \ar[d]^-{\cap [X,\partial X]} &
C^{n-*}(\widetilde{Y}) \ar[l]^-{C^{n-*}(\widetilde{f})}
\ar[d]^-{\cap [Y,\partial Y]}
\\
C_*(\widetilde{X},\widetilde{\partial X}) \ar[r]^-{C_*(\widetilde{f},\widetilde{\partial f})}&
C_*(\widetilde{Y},\widetilde{\partial Y})
}$$
The composition formula for Whitehead torsion implies
$$\rho(Y,\partial Y) = \tau\big(C_*(\widetilde{f},\widetilde{\partial f})\bigr)
+ \rho(X,\partial X) + \tau\bigl(C^{n-*}(\widetilde{f})\bigr).$$
We get from additivity and the definitions
\begin{eqnarray*}
\tau\bigl(C_*(\widetilde{f},\widetilde{\partial f})\bigr)
& = & \tau(f) - (j_{\partial Y})_*(\tau(\partial f));
\\
C^{n-*}(\widetilde{f})
& = & (-1)^n \cdot \ast\bigl(\tau(f)\bigr).
\end{eqnarray*}
Now assertion~\ref{the:prop_of_Poinc._tor:homotopy_invariance}
follows.
\\[2mm]%
\ref{the:prop_of_Poinc._tor:glueing_formula}
Define $Z$ by the pushout
\xycomsquare{\partial X}{f}{\partial Y}{i_X}{}{X}{\overline{f}}{Z}
Then $(\overline{f},f) \colon (X,\partial X) \to (Z,\partial Y)$
is a homotopy equivalence of pairs of spaces. In the sequel we treat only the case,
where $X$ and hence also $Z$ are connected, the general case follows
by inspecting the individual components. 
The pair $(\overline{f},f)$ induces a base preserving  
isomorphism of cellular $\IZ[\pi]$-chain complexes
$$C_*(\widetilde{\overline{f}},\widetilde{f}) 
\colon C_*(\widetilde{X},\widetilde{\partial X}) 
\xrightarrow{\cong}  C_*(\widetilde{Z},\widetilde{\partial Y}),$$
where we identify $\pi = \pi_1(X) = \pi_1(Z)$, $\widetilde{X} \to X$ and
$\widetilde{Z} \to Z$ are the universal coverings and $\widetilde{\partial X}$ and
$\widetilde{\partial Y}$ are obtained by restriction to $\partial X$ and $\partial Y$. 
In particular $\tau\bigl(C_*(\widetilde{\overline{f}},\widetilde{f})\bigr) = 0$.
We conclude from assertion~\ref{the:prop_of_Poinc._tor:homotopy_invariance}
and the composition formula and additivity of Whitehead torsion.
\begin{eqnarray*}
\lefteqn{\rho(Z,\partial Y) - \rho(X,\partial X)}
& & 
\\ 
& = & 
\tau(\overline{f}) + (-1)^n \cdot \ast\bigl(\tau(\overline{f})\bigr) 
- (j_{\partial Y})_*\tau(f)
\\
& = & 
\bigl(\tau(\overline{f}) - (j_{\partial Y})_*(\tau(f))\bigr) 
+ (-1)^n \cdot  \ast\bigl(\tau(\overline{f}) - (j_{\partial Y})_*(\tau(f)) 
+ (j_{\partial Y})_*(\tau(f))\bigr)
\\
& = &
\tau\bigl(C_*(\widetilde{f},\widetilde{\overline{f}})\bigr) 
+ (-1)^n \cdot \ast\bigl(\tau(C_*(\widetilde{f},\widetilde{\overline{f}}) 
+ (j_{\partial Y})_*(\tau(f))\bigr)
\\
& = & 
(-1)^n \cdot \ast\circ (j_{\partial Y})_*\bigl(\tau(f)\bigr).
\end{eqnarray*}
There is an obvious 
homeomorphism $X \cup_f Y \xrightarrow{\cong} Z \cup_{\partial Y} Y$. Hence it remains
to show
\begin{eqnarray}
\rho(Z \cup_{\partial Y} Y) & = &
(-1)^n \cdot \ast \circ \bigl(j_Z)_*\bigl(\rho(Z,\partial Y)\bigr)
+ (j_Y)_*\bigl(\rho(Y,\partial Y)\bigr),
\label{remains_to_show}
\end{eqnarray}
where  $j_Z \colon Z \to Z \cup_{\partial Y} Y$ and 
$j_Y \colon Y \to Z \cup_{\partial Y} Y$ are the canonical inclusions.
In the following let $\widetilde{Z\cup_{\partial Y} Y} \to Z\cup_{\partial Y} Y$ be the
universal covering. Denote by $\widetilde{X} \to
X$, $\widetilde{Y} \to Y$, $\widetilde{\partial Y} \to \partial Y$
the restriction of it to $Y$, $X$ and $\partial Y$. Notice
that the these are not necessarily the universal coverings. 
By excision we obtain an isomorphism
$$H_n(Z,\partial Y;\IZ^w) \oplus H_n(Y,\partial Y;\IZ^w) 
\xrightarrow{\cong} H_n(Z \cup_{\partial Y} Y,\partial Y;\IZ^w).$$
The boundary homomorphisms $H_n(Y,\partial Y;\IZ^w) 
\to H_{n-1}(\partial Y;\IZ^w)$
and $H_n(Z,\partial Y;\IZ^w) \to H_{n-1}(\partial Y;\IZ^w)$ are injective
and we can arrange  such that $[Y,\partial Y]$ and $[Z,\partial Y]$
are mapped to $[\partial Y]$. The Mayer-Vietoris sequence yields an exact sequence
$$0 \to H_n(Z \cup_{\partial Y} Y;\IZ^w) 
\to H_n(Y;\partial Y;\IZ) \oplus H_n(Z,\partial Y;\IZ^w)
\to H_{n-1}(\partial Y;\IZ^w).$$
Let $[Z \cup_{\partial Y} Y]$ be the unique element in
$H_n(Z \cup_{\partial Y} Y;\IZ^w)$ which is mapped to
$\bigl([Z,\partial Y], [Y,\partial Y]\bigr)$. 
Then $[Z \cup_{\partial Y} Y]$ generates the infinite
cyclic group $H_n(Z \cup_{\partial Y} Y;\IZ)$ and we obtain a commutative diagram of
based free $\IZ \pi$-chain complexes whose vertical maps are the Poincar\'e duality
chain homotopy equivalences and whose rows are based exact sequences of finite based free
$\IZ \pi$-chain complexes.
$$
\xymatrix{
0 \ar[r] 
&
C^{n-*}(\widetilde{Z},\widetilde{\partial Y}) \ar[r] \ar[d]_{-\cap [Z,\partial Y]}
&
C^{n-*}(\widetilde{Z \cup_{\partial Y} Y}) \ar[r] \ar[d]_{-\cap [Z \cup_{\partial Y} Y]}
&
C^{n-*}(\widetilde{Y}) \ar[r] \ar[d]_{-\cap [Y,\partial Y]}
& 0
\\
0 \ar[r] 
&
C_*(\widetilde{Z}) \ar[r]
&
C_*(\widetilde{Z \cup_{\partial Y} Y}) \ar[r] 
&
C_*(\widetilde{Y},\widetilde{\partial Y}) \ar[r] 
& 0
}
$$
Additivity of the Whitehead torsion implies
that Whitehead torsion of the middle vertical arrow is the sum of the Whitehead torsions
of the left and of the right vertical arrow. The Whitehead torsion of the right vertical 
arrow is by definition $(j_Y)_*\bigl(\rho(Y,\partial Y)\bigr)$ 
and the Whitehead torsion of the
middle arrow is $\rho(Z \cup_{\partial Y} Y)$. If we apply $C^{n-*}$ to the left
arrow, we obtain 
$$-\cap [Z,\partial Y] \colon C^{n-*}(\widetilde{Z})
\to C_*(\widetilde{Z},\widetilde{\partial Y})$$ 
Hence  Whitehead torsion of the left arrow is 
$(-1)^n \cdot \ast \circ (j_Z)_*\bigl(\tau(Z,\partial Y)\bigr)$. 
This proves~\eqref{remains_to_show} and hence
assertion~\ref{the:prop_of_Poinc._tor:glueing_formula}.
\\[2mm]\ref{the:prop_of_Poinc._tor:product_formula} follows from the product
formula of Whitehead torsion and the K\"unneth isomorphism
\begin{eqnarray*}
C_*(\widetilde{X},\widetilde{\partial X}) \otimes
C_*(\widetilde{Y},\widetilde{\partial Y})
& \xrightarrow{\cong} &
C_*((\widetilde{X},\widetilde{\partial X})  \times
(\widetilde{Y},\widetilde{\partial Y}));
\\
C^{m-*}(\widetilde{X}) \otimes C^{n-*}(\widetilde{Y})
&\xrightarrow{\cong} &
C^{n+m-*}(\widetilde{X \times Y}),
\end{eqnarray*}
which are compatible with the various Poincar\'e chain duality maps.
\\[2mm]\ref{the:prop_of_Poinc._tor:triviality_for_manifolds}  
Kirby-Siebenmann~\cite[Essay~III, Theorem~5.13 on page~136]{Kirby-Siebenmann(1977)})
prove that there is a simple homotopy equivalence of pairs 
$(M,\partial M) \to (X,\partial X)$ for a simple finite Poincar\'e pair
$(X,\partial X)$.
This finishes the proof of
Theorem~\ref{the:prop_of_Poinc._tor}.
\end{proof}

Denote by $\widehat{H}^{n}(\IZ/2,\Wh(\pi))$ the Tate homology of
$\IZ/2$ with coefficients in $\Wh(\pi)$ with respect to the
involution $\ast$ introduced above. Explicitly
$$\widehat{H}^{n}(\IZ/2,\Wh(\pi))=\{x \in \Wh(\pi)\mid \ast x = (-1)^n\cdot x\}/
\{y + (-1)^n\cdot \ast y \mid y \in \Wh(\pi))\}.$$
Let $X$ be a space which has the homotopy type of a
finite $n$-dimensional Poincar\'e complex.
Let $f\colon X \to Y$ be any homotopy equivalence to a
finite $n$-dimensional Poincar\'e complex $Y$.
The Poincar\'e torsion $\rho(Y) \in \Wh(\pi(Y))$ satisfies
$\rho(Y) = (-1)^n \cdot \ast\rho(Y)$ by
Theorem~\ref{the:prop_of_Poinc._tor}~\ref{the:prop_of_Poinc._tor:behaviour_under_involution}
and hence defines a class in
$\widehat{H}^{n}(\IZ/2,\Wh(\pi(Y)))$. Denote by
$$\widehat{\rho}(X) \in \widehat{H}^n\bigl(\IZ/2;\Wh(\pi(X))\bigr)$$
the image of $\rho(Y)$ under the bijection
$\widehat{H}^n\bigl(\IZ/2;\Wh(\pi(Y))\bigr) \xrightarrow{\cong}
\widehat{H}^n\bigl(\IZ/2;\Wh(\pi(X))\bigr)$ induced by $f^{-1}$.
This is independent of the choice of
$f\colon X \to Y$ by
Theorem~\ref{the:prop_of_Poinc._tor}~\ref{the:prop_of_Poinc._tor:homotopy_invariance}.

\begin{definition} \label{def:Tate-Poincare_torsion}
Given a space $X$ of the homotopy type of a
finite $n$-dimensional Poincar\'e $CW$-complex we call
$$\widehat{\rho}(X) \in \widehat{H}^n\bigl(\IZ/2;\Wh(\pi(X))\bigr)$$
the \emph{Tate-Poincar\'e} torsion of $X$.
\end{definition}

We conclude from Theorem~\ref{the:prop_of_Poinc._tor}
and the definitions.

\begin{theorem} \label{the:_prop_of_Tate-Poinc._tor}
\begin{enumerate}

\item \label{the:_prop_of_Tate-Poinc._tor:homotopy_invariance}
If $f\colon X \to Y$
is a homotopy equivalence of spaces of the homotopy type
of finite $n$-dimensional Poincar\'e complexes, then the induced
isomorphism $\widehat{H}^n\bigl(\IZ/2,\Wh(\pi(X))\bigr) \xrightarrow{\cong}
\widehat{H}^n\bigl(\IZ/2,\Wh(\pi(Y))\bigr)$ maps $\widehat{\rho}(X)$ to
$\widehat{\rho}(Y)$;

\item \label{the:_prop_of_Tate-Poinc._tor:bordism_invariance}
Let $(X,\partial X)$ be a Poincar\'e pair. Let $j_{\partial M} \colon
\partial X \to X$ be the inclusion. Then the map
$$(j_{\partial X})_* \colon \widehat{H}^n\bigl(\IZ/2,\Wh(\pi(\partial X))\bigr)
\to \widehat{H}^n\bigl(\IZ/2,\Wh(\pi(X))\bigr)$$ 
induced by $j_{\partial M}$ sends $\widehat{\rho}(\partial X)$ to zero;
  
\item \label{the:_prop_of_Tate-Poinc._tor:product_formula}
Let $X$ resp. $Y$ be a space of the homotopy type of
a connected $m$- resp. $n$- dimensional finite Poincar\'e complex.
Then $X \times Y$ has the homotopy type of a
connected finite $(m+n)$-dimensional Poincar\'e $CW$-complex and
$$\widehat{\rho}(X \times Y) = \chi(X) \cdot (k_Y)_*\bigl(\widehat{\rho}(Y)\bigr)
+ \chi(Y)\cdot (k_X)_*\bigl(\widehat{\rho}(X)\bigr),$$
where $k_X\colon X \to X \times Y$ and $k_Y\colon Y \to X \times Y$ are the
inclusions;

\item \label{the:_prop_of_Tate-Poinc._tor:triviality_for_manifolds}
If $X$ has the homotopy type of an $n$-dimensional closed topological manifold,
then $X$ is homotopy equivalent to a simple $n$-dimensional Poincar\'e complex
and in particular
$$\widehat{\rho}(X) = 0;$$

\item \label{the:_prop_of_Tate-Poinc._tor:homotopy_versus_simple}
Let $X$ be an $n$-dimensional Poincar\'e complex. Then it  is homotopy
equivalent to a simple $n$-dimensional Poincar\'e complex if and only if
$\widehat{\rho}(X) = 0$ holds in $\widehat{H}^n(\IZ/2,\Wh(X))$.

\end{enumerate}
\end{theorem}
\begin{proof}\ref{the:_prop_of_Tate-Poinc._tor:homotopy_invariance}
This follows from
Theorem~\ref{the:prop_of_Poinc._tor}~\ref{the:prop_of_Poinc._tor:homotopy_invariance}.
\\[2mm]\ref{the:_prop_of_Tate-Poinc._tor:bordism_invariance}
This follows from  Theorem~\ref{the:prop_of_Poinc._tor}~%
\ref{the:prop_of_Poinc._tor:behaviour_under_involution}.
\\[2mm]\ref{the:_prop_of_Tate-Poinc._tor:product_formula}
This follows from
Theorem~\ref{the:prop_of_Poinc._tor}~\ref{the:prop_of_Poinc._tor:product_formula}.
\\[2mm]\ref{the:_prop_of_Tate-Poinc._tor:triviality_for_manifolds}
This follows from
Theorem~\ref{the:prop_of_Poinc._tor}~\ref{the:prop_of_Poinc._tor:triviality_for_manifolds}
and assertion~\ref{the:_prop_of_Tate-Poinc._tor:homotopy_invariance}.
\\[2mm]\ref{the:_prop_of_Tate-Poinc._tor:homotopy_versus_simple}
Since $\widehat{\rho}(X) = 0$ in $\widehat{H}^n(\IZ/2,\Wh(X))$, we can find
$y \in \Wh(X)$ with $-\rho(X) = y + (-1)^n \ast y$.
Choose a finite $CW$-complex $Y$ together with a homotopy equivalence
$f \colon Y \to X$ satisfying $\tau(f) = y \in \Wh(X)$.
Then we conclude $\rho(Y) = 0$ from
Theorem~\ref{the:prop_of_Poinc._tor}~\ref{the:prop_of_Poinc._tor:homotopy_invariance}.
\end{proof}

\begin{remark} \label{rem:obstruction_to_be_a-manifold.}
Let $f \colon  M \to N$ be a map of closed manifolds. 
Suppose that $M$ and $N$ are connected.
Suppose that the homotopy fiber $\hofib(f)$ has the homotopy type of a finite
$CW$-complex. Then we have the fibration of spaces of the homotopy type of finite
$CW$-complexes
$$\hofib(f) \to \FIB(f) \to N$$
such that the total space has the homotopy type of a finite Poincar\'e complex and
the base space is a finite Poincar\'e complex. 
We conclude from~\cite{Gottlieb(1979)} that also the homotopy fiber has the homotopy
type of a finite Poincar\'e complex. 
Hence we can define
\begin{eqnarray}
\widehat{\rho(f)} &  = & \widehat{\rho}\bigl(\hofib(f)\bigr) 
\in \widehat{H}^n\bigl(\IZ/2,\Wh(\pi(\hofib(f))\bigr).
\label{Poincare-torsion_obstruction}
\end{eqnarray}

Suppose that $f$ is homotopic to a map $p \colon M \to N$ which is the projection of
a locally trivial fiber bundle with a closed manifold $F$ as fiber.
Then the homotopy fiber $\hofib(f)$ of $f$
is homotopy equivalent to $F$. Theorem~\ref{the:_prop_of_Tate-Poinc._tor}
implies 
$$\widehat{\rho}(\hofib(f)) = 0.$$
Hence we have besides the obstructions appearing in 
Definition~\ref{def:fiber_torsion_obstruction}
another torsion obstruction for
$f$ to be homotopic to a bundle projection.
\end{remark}

\typeout{-----------------------  Section 11 ------------------------}

\section{Connection to the parametrized $A$-theory characteristic}
\label{sec:Connection_to_the_parametrized_A-theory_characteristic}

The element $\Theta(f)$ defined here is a shadow of the parametrized
$A$-theory characteristic, as defined by Dwyer-Weiss-Williams
\cite{Dwyer-Weiss-Williams(2003)}. In this section we give a sketch of
the relationship.

The parametrized $A$-theory characteristic $\chi(p)$ is defined for
any fibration $p\colon E\to B$ with homotopy finitely dominated fibers
and can be understood as a section of the fibration obtained from $p$ by
applying the (connective) $A$-theory functor fiberwise. We are going
to write
\[\chi(p)\in \pi_0\sections{A_B(E)}{B} =: H^0(B;A(F_b)),\]
thinking of this as the zeroth cohomology of $B$ with twisted
coefficients in the spectrum $A(F_b)$ (where $F_b$ denotes the fiber
of $p$ over $b$).

The natural transformation $A(X)\to\Wh^{PL}(X)$ induces a map
$H^0(B;A(F_b))\to H^0(B;\Wh^{PL}(F_b))$; denote the image of $\chi(p)$
by $\Wall(p)$, the parametrized Wall obstruction. Dwyer-Weiss-Williams
show the following:

\begin{theorem}
  The fibration $p$ is fiber homotopy equivalent to a bundle of
  compact topological manifolds (possibly with boundary) if and only
  if $\Wall(p)=0$.
\end{theorem}

Therefore, if a map $f\colon M\to B$ between manifolds is homotopic to
a fiber bundle, then $\Wall(p)$ is defined and vanishes, 
with $p\colon E\to B$ the fibration associated to $f$. 
There is an Atiyah-Hirzebruch type spectral sequence
\[E_2^{pq}=H^p(B;\pi_{-q}\Wh^{PL}(F_b)) \: \Longrightarrow \:
H^{p+q}(B;\Wh^{PL}(F_b)),\] where the cohomology on the left hand side
is ordinary cohomology with twisted coefficients in the system
$\{b\mapsto \pi_{-q}\Wh^{PL}(F_b)\}$. Denote by $P_1\Wh^{PL}(F_b)$ the
first Postnikov approximation of $\Wh^{PL}(F_b)$, such that 
$\pi_n P_1\Wh^{PL}(F_b)=0$ for $n\geq 2$. The Atiyah-Hirzebruch spectral
sequence reduces to the exact sequence
\[0\to H^1(B;\pi_1\Wh^{PL}(F_b))\to H^0(B;P_1\Wh^{PL}(F_b))\to
H^0(B;\pi_0\Wh^{PL}(F_b))\to 0.\]
Denote by $P_1\Wall(p)$ the image of $\Wall(p)$ in $H^0(B;P_1\Wh^{PL}(F_b))$.

\begin{proposition}\label{prop:A-theory}
  \begin{enumerate}
  \item \label{prop:A-theory:image}
  The image of $P_1\Wall(p)$ in
    \[H^0(B;\pi_0\Wh^{PL}(F_b))\cong \bigoplus_{[b]\in\pi_0 B} \tilde
    K_0(\IZ[\pi_1(F_b,b)])^{\pi_1(B,b)}\] consists of the Wall
    obstructions of the fibers over every path component of $B$;
  \item \label{prop:A-theory:theta} 
    Suppose that all the fibers have the homotopy type of a finite
    $CW$-complex. The lift of $P_1\Wall(p)$ to
    $H^1(B;\pi_1\Wh^{PL}(F_b))$ maps to $\Theta(p)$ under the
    map
    \[H^1(B;\pi_1\Wh^{PL}(F_b))\to H^1(B;\pi_1\Wh^{PL}(E))\cong
    H^1(B;\Wh(\pi(E)))\]
    induced by the inclusion.
  \end{enumerate}
\end{proposition}
\begin{proof}%
\ref{prop:A-theory:image} This assertion  mainly depends
on the fact that the path component of the unparametrized $A$-theory
characteristic gives the unreduced Wall obstruction (which follows
rather easily from the linearization map to $K$-theory). 
\\%
\ref{prop:A-theory:theta} One needs to show 
that a simple structure on a space $X$ is the
same thing as a (homotopy class of a) lift of $\chi(X)$ to an
``excisive characteristic'' $\chi^\%(X)\in A^\%(X)$, and that the
naturality of the $A$-theory characteristic for homotopy equivalences
allows to describe the Whitehead torsion with respect to these lifts.

Then observe that Waldhausen's description of the $A$-theory assembly
map \cite{Waldhausen(1985)} defines a canonical excisive
characteristic for finite $CW$ complexes. The equivalence of the
algebraic and the geometric definition of the Whitehead group \cite[\S
21]{Cohen(1973)} implies that the corresponding simple structure is
just the canonical one. Once one has identified Waldhausen's excisive
characteristic with the excisive characteristic $\chi^\%(X)$ defined
by Dwyer-Weiss-Williams for compact ENRs $X$ (in the case where both
are defined), assertion~\ref{prop:A-theory:theta} 
follows by the construction of the short exact sequence.
\end{proof}

\begin{remark}
  In unpublished work~\cite{Weiss-Williams(2009autoIII)},
  Weiss-Williams considerably strengthen the $A$-theory
  characteristic to a so-called $LA$-theory characteristic, defined for a finitely
  dominated Poincar\'e duality space. There is a corresponding parametrized version taking
  values in $H^0(B;LA(F_b))$ in our notation. Let $p$ be a fibration with Poincar\'e
  duality spaces as fibers, such that the dimension of the base is small compared to the
  formal dimension of the fibers. The image of the parametrized $LA$-theory characteristic
  of $p$ in the cofiber of the $LA$-theoretic assembly map is ``almost'' the total
  obstruction for $p$ to be fiber homotopy equivalent to a fiber bundle of \emph{closed}
  topological manifolds.
\end{remark}

\typeout{-----------------------Section 12---------------------}

\section{Some questions}
\label{sec:Some_questions}

  Let $f \colon M \to N$ and $g \colon N \to B$
  be maps of closed path-connected manifolds. Assume that the homotopy
  fiber of both $f$ and $g$ has the homotopy type of a finite
  $CW$-complex.  Then the same is true for the composite $g \circ f$
  since there is a fibration $\hofib(f) \to \hofib(g \circ f) \to
  \hofib(g)$. So the elements $\Theta(f) \in H^1\bigl(N,\Wh(\pi(M))\bigr)$,
  $\Theta(g) \in H^1\bigl(N,\Wh(\pi(N))\bigr)$ and $\Theta(g \circ f) \in
  H^1\bigl(B;\Wh(\pi(M))\bigr)$ are defined. 

\begin{question}\label{que:theta}
What is the relation between $\Theta(g \circ f)$, $\Theta(f)$
and $\Theta(g)$?
\end{question}

\begin{question}\label{que:B_apherical}
If $N$ is aspherical, what are the other obstructions besides the
torsion obstructions presented in this paper for $f$ to be homotopic to a 
bundle projection?
\end{question}

Notice that in the case $B = S^1$ there are no other obstructions
because of Theorem~\ref{the:Comparison_with_Farrell}.

\begin{question}\label{que:homotopy_fiber}
Suppose that $M$ and $N$ are aspherical. Is then the homotopy fiber
  a closed manifold?
\end{question}

The question may have a positive answer in favorite circumstances because of the
following remarks. Suppose that the Farrell-Jones
Conjecture for algebraic $K$- and $L$-theory with arbitrary coefficients 
is true for the fundamental group of $E$. (This is known to be true
for a large class of groups.) Assume that the difference 
of the dimensions of $E$ and $B$ is at least six. Moreover, assume that
the resolution obstruction of Quinn 
(see~\cite{Quinn(1987_resolution)}) vanishes for all aspherical closed
ANR-homology manifolds. (There is no counterexample to this assumption 
known to the authors.) Then one can deduce that the homotopy fiber
is homotopy equivalent to a closed topological manifold and this 
closed topological manifold is unique up to homeomorphism
(see~\cite{Bryant-Ferry-Mio-Weinberger(1996)}).

\begin{question} \label{que:B_and_E_aspherical}
Suppose that $M$ and $N$ are aspherical.
Are there any obstructions for $p$ being homotopic to a bundle
projection of a locally trivial bundle, or for weaker notions,
such as block bundles?
\end{question}

Quinn developed a technique addressing the block bundle case of this question in~\cite[Section~1]{Quinn(1970)}. Using Quinn's technique a partial result on the block bundle question was obtained in~\cite[Theorem~10.7]{Farrell-Jones(1989)}.

\typeout{-----------------------  References ------------------------}



\begin{thebibliography}{10}

\bibitem{Bryant-Ferry-Mio-Weinberger(1996)}
J.~Bryant, S.~Ferry, W.~Mio, and S.~Weinberger.
\newblock Topology of homology manifolds.
\newblock {\em Ann. of Math. (2)}, 143(3):435--467, 1996.

\bibitem{Chapman(1974)}
T.~A. Chapman.
\newblock Topological invariance of {W}hitehead torsion.
\newblock {\em Amer. J. Math.}, 96:488--497, 1974.

\bibitem{Cohen(1973)}
M.~M. Cohen.
\newblock {\em A course in simple-homotopy theory}.
\newblock Springer-Verlag, New York, 1973.
\newblock Graduate Texts in Mathematics, Vol. 10.

\bibitem{Dwyer-Weiss-Williams(2003)}
W.~Dwyer, M.~Weiss, and B.~Williams.
\newblock A parametrized index theorem for the algebraic {$K$}-theory {E}uler
  class.
\newblock {\em Acta Math.}, 190(1):1--104, 2003.

\bibitem{Farrell(1971)}
F.~T. Farrell.
\newblock The obstruction to fibering a manifold over a circle.
\newblock {\em Indiana Univ. Math. J.}, 21:315--346, 1971/1972.

\bibitem{Farrell-Hsiang(1970)}
F.~T. Farrell and W.-C. Hsiang.
\newblock A formula for ${K}\sb{1}{R}\sb{\alpha }\,[{T}]$.
\newblock In {\em Applications of Categorical Algebra (Proc. Sympos. Pure
  Math., Vol. XVII, New York, 1968)}, pages 192--218. Amer. Math. Soc.,
  Providence, R.I., 1970.

\bibitem{Farrell-Jones(1989)}
F.~T. Farrell and L.~E. Jones.
\newblock A topological analogue of {M}ostow's rigidity theorem.
\newblock {\em J. Amer. Math. Soc.}, 2(2):257--370, 1989.

\bibitem{Gottlieb(1979)}
D.~H. Gottlieb.
\newblock Poincar\'e duality and fibrations.
\newblock {\em Proc. Amer. Math. Soc.}, 76(1):148--150, 1979.

\bibitem{Kirby-Siebenmann(1977)}
R.~C. Kirby and L.~C. Siebenmann.
\newblock {\em Foundational essays on topological manifolds, smoothings, and
  triangulations}.
\newblock Princeton University Press, Princeton, N.J., 1977.
\newblock With notes by J.~Milnor and M.~F.~Atiyah, Annals of Mathematics
  Studies, No. 88.

\bibitem{Korzieniewski(2005)}
A.~Korzeniewski.
\newblock Absolute {W}hitehead torsion.
\newblock {\em Geom. Topol.}, 11:215--249, 2007.

\bibitem{Lueck(1984)}
W.~L{\"u}ck.
\newblock {\em Eine allgemeine algebraische Beschreibung des Transfers f{\"u}r
  Faserungen auf projektiven Klassengruppen und Whiteheadgruppen}.
\newblock PhD thesis, Universit{\"a}t G{\"o}ttingen, 1984.

\bibitem{Lueck(1986)}
W.~L{\"u}ck.
\newblock The transfer maps induced in the algebraic ${K}\sb 0$-and ${K}\sb
  1$-groups by a fibration. {I}.
\newblock {\em Math. Scand.}, 59(1):93--121, 1986.

\bibitem{Lueck(1987)}
W.~L{\"u}ck.
\newblock The transfer maps induced in the algebraic ${K}\sb 0$- and ${K}\sb
  1$-groups by a fibration. {I}{I}.
\newblock {\em J. Pure Appl. Algebra}, 45(2):143--169, 1987.

\bibitem{Lueck(1989)}
W.~L{\"u}ck.
\newblock {\em Transformation groups and algebraic ${K}$-theory}, volume 1408
  of {\em Lecture Notes in Mathematics}.
\newblock Springer-Verlag, Berlin, 1989.

\bibitem{Lueck-Ranicki(1992)}
W.~L{\"u}ck and A.~A. Ranicki.
\newblock Surgery obstructions of fibre bundles.
\newblock {\em J. Pure Appl. Algebra}, 81(2):139--189, 1992.

\bibitem{Milnor(1966)}
J.~Milnor.
\newblock Whitehead torsion.
\newblock {\em Bull. Amer. Math. Soc.}, 72:358--426, 1966.

\bibitem{Quinn(1970)}
F.~Quinn.
\newblock A geometric formulation of surgery.
\newblock In {\em Topology of Manifolds (Proc. Inst., Univ. of Georgia, Athens,
  Ga., 1969)}, pages 500--511. Markham, Chicago, Ill., 1970.

\bibitem{Quinn(1987_resolution)}
F.~Quinn.
\newblock An obstruction to the resolution of homology manifolds.
\newblock {\em Michigan Math. J.}, 34(2):285--291, 1987.

\bibitem{Ranicki(1980b)}
A.~A. Ranicki.
\newblock The algebraic theory of surgery. {I}{I}. {A}pplications to topology.
\newblock {\em Proc. London Math. Soc. (3)}, 40(2):193--283, 1980.

\bibitem{Siebenmann(1970_total)}
L.~C. Siebenmann.
\newblock A total {W}hitehead torsion obstruction to fibering over the circle.
\newblock {\em Comment. Math. Helv.}, 45:1--48, 1970.

\bibitem{Steenrod(1967)}
N.~E. Steenrod.
\newblock A convenient category of topological spaces.
\newblock {\em Michigan Math. J.}, 14:133--152, 1967.

\bibitem{Steimle(2007)}
W.~Steimle.
\newblock Whitehead-{T}orsion und {F}aserungen.
\newblock Diplomarbeit, Arxiv:math.GT/0706.397v1, 2007.

\bibitem{Switzer(1975)}
R.~M. Switzer.
\newblock {\em Algebraic topology---homotopy and homology}.
\newblock Springer-Verlag, New York, 1975.
\newblock Die Grundlehren der mathematischen Wissenschaften, Band 212.

\bibitem{Waldhausen(1985)}
F.~Waldhausen.
\newblock Algebraic ${K}$-theory of spaces.
\newblock In {\em Algebraic and geometric topology (New Brunswick, N.J.,
  1983)}, pages 318--419. Springer-Verlag, Berlin, 1985.

\bibitem{Wall(1999)}
C.~T.~C. Wall.
\newblock {\em Surgery on compact manifolds}.
\newblock American Mathematical Society, Providence, RI, second edition, 1999.
\newblock Edited and with a foreword by A. A. Ranicki.

\bibitem{Weiss-Williams(2009autoIII)}
M.~Weiss and B.~Williams.
\newblock Automorphisms of manifolds and algebraic {$K$}-theory: {III}.
\newblock Preprint, Aberdeen,
  http://www.maths.abdn.ac.uk/~mweiss/pubtions.html, 2009.

\bibitem{Whitehead(1978)}
G.~W. Whitehead.
\newblock {\em Elements of homotopy theory}, volume~61 of {\em Graduate Texts
  in Mathematics}.
\newblock Springer-Verlag, New York, 1978.

\end{thebibliography}

\end{document}